\documentclass[12pt]{article}
\usepackage[margin = .991in]{geometry}
\usepackage{amsthm,amssymb,amsmath,graphicx,cite, color,mathrsfs}
\usepackage{algpseudocode,algorithm,algorithmicx}
\usepackage{mathtools}
\mathtoolsset{showonlyrefs}
\usepackage{enumitem}
\usepackage[square,numbers]{natbib}

\newcommand{\ip}[2]{\langle #1 , #2 \rangle}    
\newcommand{\dmin}{\displaystyle\min}
\newcommand{\dmax}{\displaystyle\max}
\newcommand{\dsum}{\displaystyle\sum}
\newcommand{\R}{{\mathbb R}}
\renewcommand{\S}{{\mathfrak S}}
\newcommand{\sJ}{\mathcal S}
\newcommand{\D}{{\mathcal D}}
\renewcommand{\L}{{\mathcal L}}

\newcommand{\linspan}{{\mathrm{span}}}

\newcommand{\conv}{{\mathrm{conv}}}
\newcommand{\dom}{{\mathrm{dom}}}
\newcommand{\dist}{{\mathrm{dist}}}

\newcommand{\graph}{{\mathrm{graph}}}
\newcommand{\ext}{{\mathrm{ext}}}
\newcommand{\vertiii}[1]{\| #1 \|}
\newcommand{\vertii}[1]{{\vert\kern-0.25ex\vert\kern-0.25ex\vert #1     \vert\kern-0.25ex\vert\kern-0.25ex\vert}}
\usepackage{enumitem}

\newcommand{\1}{{\mathbf 1}}
\newcommand{\ri}{{\mathrm{ri}}}

\DeclareMathOperator{\Image}{Im}

\newtheorem{lemma}{Lemma}
\newtheorem{theorem}{Theorem}
\newtheorem{proposition}{Proposition}

\newtheorem{corollary}{Corollary}

\def\transp{^{\text{\sf T}}}
\newcommand{\T}{\mathcal T}
\newcommand{\J}{\mathcal J}
\newcommand{\F}{\mathcal F}
\renewcommand{\H}{\mathcal H}

\newcommand{\I}{\mathcal I}

\newcommand{\matr}[1]{\begin{bmatrix} #1 \end{bmatrix}}    

\title{An algorithm to compute the Hoffman constant of a system of linear constraints}
\author{Javier Pe\~na\thanks{Tepper School of Business,
Carnegie Mellon University, USA, {\tt jfp@andrew.cmu.edu}} \and Juan Vera\thanks{Department of Econometrics and Operations Research,
Tilburg University, The Netherlands, {\tt j.c.veralizcano@uvt.nl}}
\and Luis F. Zuluaga\thanks{Department of Industrial and Systems Engineering, Lehigh University, USA, {\tt luis.zuluaga@lehigh.edu}}}

\begin{document}

\maketitle

\begin{abstract}
We propose a combinatorial algorithm to compute the Hoffman constant of a system of linear equations and inequalities.  The algorithm is based on a characterization of the Hoffman constant as the largest of a finite canonical collection of easy-to-compute Hoffman constants.  Our algorithm and characterization extend to  the more general context where some of the constraints are easy to satisfy as in the case of box constraints.  
We highlight some natural connections between our characterizations of the Hoffman constant and Renegar's distance to ill-posedness for systems of linear constraints.

\end{abstract}


\section{Introduction}
\label{sec.intro}

A classical result of~\citet{Hoff52} shows that the distance between a point $u \in \R^n$ and a non-empty polyhedron $P_{A,b}:=\{x \in \R^n: Ax \le b\}$ can be bounded above in terms of the size of the {\em residual} vector $(Au-b)_+ := \max(0,Au-b)$.  More precisely, for $A \in \R^{m\times n}$ there exists a {\em Hoffman constant} $H(A)$ that depends only on $A$ such that for all $b\in\R^m$ with $P_{A,b} \ne \emptyset$ and all $u\in \R^n$,
\begin{equation}
\label{eq:simple_erro}
\dist(u,P_{A,b}) \le H(A) \cdot \|(Au-b)_+\|.
\end{equation}
Here $\dist(u,P_{A,b}) := \min\{\|u-x\|: x\in P_{A,b}\}.$
The bound~\eqref{eq:simple_erro}  is a type of {\em error bound} for the system of inequalities $Ax \le b$, that is, an inequality bounding the distance from a point  $u\in \R^n$ to a nonempty {\em solution set} in terms of a measure of the {\em error} or {\em residual} of the point $u$.  The Hoffman bound~\eqref{eq:simple_erro} and more general error bounds play a fundamental role in mathematical programming~\cite{Nguy17,Pang97,ZhouS17}.  In particular, Hoffman bounds as well as other related error bounds are instrumental in establishing convergence properties of a variety of  algorithms~\cite{BeckS15,Garb18,GutmP18,LacoJ15,LeveL10,LuoT93,NecoNG18,PenaR16,WangL14}.  Hoffman bounds are also used to measure the 
optimality  and  feasibility  of  a  point  generated  by  rounding
an  optimal  point  of  the continuous  relaxation  of  a  mixed-integer  linear or quadratic
optimization problem~\citep{stein2016,granot1990}. Furthermore, Hoffman bounds are used in 
sensitivity analysis~\citep{jourani2000}, and to design solution methods for non-convex quadratic programs~\citep{xia2015}.

The computational task of calculating or even estimating the constant $H(A)$ is known to be notoriously challenging~\cite{KlatT95}.  The following characterization of $H(A)$ from~\cite{guler1995,KlatT95,WangL14} is often used in the optimization literature
\begin{equation}\label{eq.popular}
H(A) = \max_{J\subseteq \{1,\dots,m\}\atop A_J \text{ full row rank} }  \frac{1}{\dmin_{v\in \R^J_+, \|v\|^*=1}\|A_J\transp v\|^*}.
\end{equation}
In~\eqref{eq.popular} and throughout the paper $A_J \in \R^{J\times n}$ denotes the submatrix of $A\in \R^{m\times n}$ obtained by selecting the rows in $J \subseteq \{1,\dots,m\}.$
A naive attempt to use~\eqref{eq.popular} to compute or estimate $H(A)$ is evidently non-viable because, in principle, it requires scanning an enormous number of sets $J\subseteq \{1,\dots,m\}$.  A major limitation of~\eqref{eq.popular} is that it does not reflect the fact that the tractability of computing $H(A)$ may depend on certain structural features of $A$.  For instance, the computation of the Hoffman constant $H(A)$ is manageable when the set-valued mapping $x \mapsto Ax + \R^m_+$ is surjective, that is, when 
 $A\R^n + \R^m_+ = \R^m$. In this case, as it is shown in~\cite{KlatT95,RamdP16,robinson1973}, the sharpest constant $H(A)$ satisfying~\eqref{eq:simple_erro} is 
\[
H(A) = \dmax_{y\in \R^m\atop \|y\| = 1} \min_{x\in \R^n \atop Ax\le y}\|x\| = \frac{1}{\dmin_{v\ge 0,\, \|v\|^* = 1} \|A\transp v\|^*}.
\]
This value is computable via convex optimization for suitable norms in $\R^n$ and $\R^m$.  Furthermore, when the set-valued mapping $x \mapsto Ax + \R^m_+$ is surjective, the system of linear inequalities $Ax < 0$ is {\em well-posed,} that is, it is feasible and remains feasible for small perturbations on $A$.  In this case, the value $1/H(A)$ is precisely Renegar's {\em distance to ill-posedness}~\cite{Rene95a,Rene95b} of $Ax < 0$, that is, the size of the smallest perturbation on $A$ that destroys the well-posedness of $Ax < 0$.  

We propose a combinatorial algorithm that computes the sharpest Hoffman constant $H(A)$ for any matrix $A$ by leveraging the above well-posedness property.  The algorithm is founded on the following characterization
\[
H(A) = \max_{J \in \sJ(A) }  \frac{1}{\dmin_{v\in \R^J_+, \|v\|^*=1}\|A_J\transp v\|^*},
\]
where $\sJ(A)$ is the collection of subsets $J\subseteq \{1,\dots,m\}$ such that each $x\mapsto A_Jx + \R^J_+$ is surjective.   As we detail in Section~\ref{sec.algo}, this characterization readily enables the computation of $H(A)$ by computing $\min_{v\in \R^J_+, \|v\|^*=1}\|A_J\transp v\|^*$ over a much smaller collection $\F\subseteq \sJ(A)$.   
The identification of such a collection $\F \subseteq \sJ(A)$ is the main combinatorial challenge that our algorithm tackles.

Our characterization and algorithm to compute the Hoffman constant also extend to the more general context involving both linear equations and linear inequalities and, perhaps most interestingly, to the case where some  equations or inequalities are easy to satisfy.  The latter situation arises naturally when some of the constraints are of the form $x \le u$ or $-x \le -\ell$.  Our interest in characterizing the Hoffman constant in the more general case that includes easy-to-satisfy constraints is motivated by the recent articles~\cite{BeckS15,LacoJ15,Garb18,GutmP18,PenaR16,xia2015}.  In each  of these articles, suitable Hoffman constants for systems of linear constraints that include  easy-to-satisfy constraints play a central role in establishing key properties of modern optimization algorithms.  In particular, we show that the {\em facial distance} or {\em pyramidal width} introduced in~\cite{LacoJ15,PenaR16} is precisely a Hoffman constant of this kind.  

The paper makes the following main contributions.  First, we develop a novel algorithmic approach to compute or estimate Hoffman constants   (see Algorithm~\ref{alg:bb}, Algorithm~\ref{alg:bb.v2}, and Algorithm~\ref{alg:bb.gral}). 
Second, our algorithmic developments are supported by a fresh perspective on Hoffman error bounds based on a generic Hoffman constant for  poyhedral sublinear mappings (see Theorem~\ref{thm.main}).  This perspective readily yields a characterization of the classical Hoffman constant $H(A)$ for systems of linear inequalities (see Proposition~\ref{prop.Hoffman.gral}) and a similar characterization of the Hoffman constant for systems including both linear equations and linear inequalities (see Proposition~\ref{prop.Hoffman}).  Third, we develop  characterizations of Hoffman constants in the more general context when some of the constraints are easy to satisfy (see Proposition~\ref{prop.Hoffman.gral.rest}, Proposition~\ref{prop.Hoffman.std}, Proposition~\ref{prop.equal.easy}, and Proposition~\ref{prop.facial.dist}). 
Throughout the paper we highlight the interesting and natural but somewhat overlooked connection between the Hoffman constant and Renegar's distance to ill-posedness~\cite{Rene95a,Rene95b}, which is a cornerstone of condition measures in continuous optimization.  The paper is entirely self-contained and relies only on standard convex optimization techniques.  We make extensive use of the one-to-one correspondence between the class of sublinear set-valued mappings $\Phi:\R^n \rightrightarrows \R^m$ and the class of convex cones $K\subseteq \R^n \times \R^m$ defined via $\Phi \mapsto \graph(\Phi):=\{(x,y)\in \R^n \times \R^m: y\in \Phi(x)\}$.

Our results are related to a number of previous developments in the rich literature on error bounds~\cite{AzeC02,burke1996,guler1995,Li93,MangS87,robinson1973,VanNT09,Zali03} and on condition measures for continuous optimization~\cite{BurgC13,EpelF02,FreuV99a,FreuV03,Freu04,Lewi99,Pena00,Pena03,Rene95a,Rene95b}.  In particular, the expressions for the Hoffman constants in Proposition~\ref{prop.Hoffman.gral} and Proposition~\ref{prop.Hoffman} have appeared, albeit in slightly different form or under more restrictive conditions, in the work of  Klatte and Thiere~\cite{KlatT95}, Li~\cite{Li93},   Robinson~\cite{robinson1973}, and Wang and Lin~\cite{WangL14}.  More precisely, Klatte and Thiere~\cite{KlatT95} state and prove a version of Proposition~\ref{prop.Hoffman} under  the more restrictive assumption that $\R^n$ is endowed with the $\ell_2$ norm.  Klatte and Thiere~\cite{KlatT95} also propose an algorithm to compute the Hoffman constant which is fairly different from ours.  Li~\cite{Li93},   Robinson~\cite{robinson1973}, and Wang and Lin~\cite{WangL14} give characterizations of Hoffman constants that are equivalent to Proposition~\ref{prop.Hoffman.gral} and Proposition~\ref{prop.Hoffman} but where the maximum is taken over a different, and typically much larger, collection of index sets.
As we detail in Section~\ref{sec.Hoffman}, the expression for $H(A)$ in Proposition~\ref{prop.Hoffman.gral} can readily be seen to be at least as sharp as some bounds on~$H(A)$ derived by G\"uler et al.~\cite{guler1995} and Burke and Tseng~\cite{burke1996}.  We also note that weaker versions of Theorem~\ref{thm.main} can be obtained from  results on error bounds in Asplund spaces as those developed in the article by Van Ngai and Th{\'e}ra~\cite{VanNT09}.  Our goal to devise algorithms to compute Hoffman constants is in the spirit of and draws on the work by Freund and Vera~\cite{FreuV99a,FreuV03} to compute the distance to ill-posedness of a system of linear constraints.  Our approach to Hoffman bounds based on the correspondence  between sublinear set-valued mappings and convex cones is motivated by the work of Lewis~\cite{Lewi99}.  The characterizations of Hoffman constants when some constraints are easy to satisfy use ideas and techniques introduced by the first author in~\cite{Pena00,Pena03} and further developed by Lewis~\cite{Lewi05}.

The contents of the paper are organized as follows.  Section~\ref{sec.Hoffman} presents a characterization of the Hoffman constant $H(A)$ for $A\in \R^{m\times n}$ as the largest of a finite canonical collection of easy-to-compute Hoffman constants 
of submatrices of $A$.  We also give characterizations of similar Hoffman constants for more general cases that include both equality and inequality constraints, and where some of these constraints are easy to satisfy. 
Section~\ref{sec.algo} leverages the results of Section~\ref{sec.Hoffman} to devise an algorithm that computes the Hoffman constant $H(A)$ for $A\in \R^{m\times n}$ as well as other analogous Hoffman constants.  Section~\ref{sec.proof} contains our main theoretical result, namely a characterization of the Hoffman constant $\H(\Phi \vert \L)$ for a polyhedral sublinear mapping $\Phi:\R^n \rightrightarrows \R^m$ when the residual is known to intersect a particular linear subspace $\L\subseteq\R^m$.  The constant $\H(\Phi \vert \L)$ is the maximum of the norms of a canonical set of polyhedral sublinear mappings associated to $\Phi$ and $\L$.  
Section~\ref{sec.proof.Hoffman} presents the proofs of the main statements in Section~\ref{sec.Hoffman}.  Each of these statements is an instantiation of the generic characterization of the Hoffman constant 
$\H(\Phi \vert \L)$ for suitable choices of $\Phi$ and $\L$.  

Throughout the paper whenever we work with an Euclidean space $\R^d$, we will assume that it is endowed with a norm $\|\cdot\|$ and inner product $\ip{\cdot}{\cdot}$.  Unless we explicitly state otherwise, our results  apply to arbitrary norms.

\section{Hoffman constants for systems of linear constraints}
\label{sec.Hoffman}

This section describes a characterization for the Hoffman constant $H(A)$  in~\eqref{eq:simple_erro} for systems of linear inequalities
\[
Ax \le b.
\]
We subsequently consider analogous Hoffman constants for systems of linear equations and inequalities
\[
\begin{array}{l}
Ax = b \\
Cx \le d.
\end{array}
\]
Although the latter case with equations and inequalities subsumes the former case, for exposition purposes we discuss separately the case with inequalities only.  The notation and main ideas in this case are simpler and easier to grasp.  The crux of the characterization of $H(A)$ based on a canonical collection of submatrices of $A$ is more apparent.

We defer the proofs of all propositions in this section to Section~\ref{sec.proof.Hoffman}, where we show that they follow from a characterization of a generic Hoffman constant for polyhedral sublinear mappings (Theorem~\ref{thm.main}).  We will rely on the following terminology.  Recall that a set-valued mapping $\Phi: \R^n \rightrightarrows \R^m$ assigns a set $\Phi(x) \subseteq \R^m$ to each $x\in \R^n.$  A set-valued mapping $\Phi: \R^n \rightrightarrows \R^m$ is {\em surjective} if $\Phi(\R^n) = \bigcup_{x\in \R^n} \Phi(x)= \R^m$.  More generally, $\Phi$ is {\em relatively surjective} if $\Phi(\R^n)$ is a linear subspace.

\subsection{The case of inequalities only}
Proposition~\ref{prop.Hoffman.gral} below gives a  characterizations of the {\em sharpest} Hoffman constant $H(A)$ such that \eqref{eq:simple_erro} holds. 
The characterization is stated in terms of a canonical collection of submatrices of $A$ that define surjective sublinear mappings.

Let $A \in \R^{m\times n}$.  We shall say that a set $J\subseteq \{1,\dots,m\}$ is {\em $A$-surjective} if the set-valued mapping $x\mapsto Ax + \{s\in \R^m: s_J \ge 0\}$ is surjective.  Equivalently, $J$ is $A$-surjective if  $A_J\R^n + \R^J_+ = \R^J,$ where $A_J$ denotes the submatrix of $A$ determined by the rows in $J$.  
For $A \in \R^{m\times n}$ let $\sJ(A)$ denote the following collection of subsets of $\{1,\dots,m\}$:
\[
\sJ(A):= \{J \subseteq \{1,\dots,m\}: J \text{ is $A$-surjective}\}.
\]
For $A \in \R^{m\times n}$ let
\begin{equation}\label{eq.Hoffman.matr}
H(A):=\dmax_{J\in \sJ(A)} H_J(A),
\end{equation}
where 
\[
H_J(A):= \dmax_{v\in \R^m \atop \|v\|\le 1} \dmin_{x\in \R^n \atop A_Jx \le v_J} \|x\|
\]
for each $J\in \sJ(A)$.  By convention $H_J(A) = 0$ if $J = \emptyset$.

Observe that the set $\sJ(A)$ is independent of the particular norms in $\R^n$ and $\R^m$.  On the other hand, the values of $H_J(A), J \in \sJ(A)$ and $H(A)$ certainly depend on these norms.  The constant $H(A)$ defined in~\eqref{eq.Hoffman.matr} is the sharpest constant satisfying~\eqref{eq:simple_erro}.


\begin{proposition}\label{prop.Hoffman.gral} 
Let  $A\in \R^{m\times n}$.  Then for all $b \in \R^m$ such that $P_{A,b} := \{x\in\R^n: Ax\le b\}\ne \emptyset$ and all $u\in \R^n$
\begin{equation}\label{eq.Hoffman.bound.matr}
\dist(u,P_{A,b}) \le H(A)\cdot \dist(b,Au + \R^m_+) \le H(A)\cdot\|(Au-b)_+\|.
\end{equation}
Furthermore, the first bound is tight: If $H(A)>0$ then there exist $b\in \R^m$ such that 
$P_{A,b} \ne \emptyset$ and $u \not \in P_{A,b}$ such that 
\[
\dist(u,P_{A,b}) = H(A)\cdot \dist(b,Au + \R^m_+).
\]
\end{proposition}

The following proposition complements Proposition~\ref{prop.Hoffman.gral} and yields a procedure to compute $H_J(A)$ for $J\in \sJ(A)$.
\begin{proposition}\label{prop.Hoffman.A.surj} 
Let  $A\in \R^{m\times n}$.  Then  for all $J\in \sJ(A)$
\begin{equation}\label{eq.HA.J}
H_J(A) = \dmax_{y\in \R^m \atop \|y\|\le 1} \dmin_{x\in \R^n \atop A_Jx \le y_J} \|x\| = \dmax_{v \in \R^J_+ \atop \|A_J\transp v\|^*\le 1} \|v\|^*   = \frac{1}{\dmin_{v \in \R^J_+, \; \|v\|^* =1 } \|A_J\transp v\|^*}.
\end{equation}
If the mapping $x\mapsto Ax+\R^m_+$ is surjective then
\begin{equation}\label{eq.HA}
H(A) = \dmax_{y\in \R^m \atop \|y\|\le 1} \dmin_{x\in \R^n \atop Ax \le y} \|x\| = \dmax_{v \in \R^m_+ \atop \|A\transp v\|^*\le 1} \|v\|^*
= \frac{1}{\dmin_{v \in \R^m_+, \; \|v\|^* =1 } \|A\transp v\|^*}.
\end{equation}
\end{proposition}


\medskip

The identity~\eqref{eq.HA} in Proposition~\ref{prop.Hoffman.A.surj} has the following geometric interpretation.  By Gordan's theorem, the mapping $x \mapsto Ax + \R^m_+$ is surjective if and only if $0 \not\in \{A\transp v: v\ge 0, \, \|v\|^* = 1\}$.
When this is the case, the quantity $1/H(A)$ is precisely the distance (in the dual norm $\|\cdot\|^*$) from the origin to $\{A\transp v: v\ge 0, \, \|v\|^* = 1\}$.  The latter quantity in turn equals the {\em distance to non-surjectivity} of the mapping $x \mapsto Ax + \R^m_+$, that is, the norm of the smallest perturbation matrix $\Delta A \in \R^{m\times n}$ such that $x \mapsto (A+\Delta A)x + \R^m_+$ is not surjective as it is detailed in~\cite{Lewi99}.  This distance to non-surjectivity is the same as Renegar's {\em distance to ill-posedness} of the system of linear inequalities $Ax < 0$ defined by $A$.  The distance to ill-posedness provides the main building block for Renegar's concept of {\em condition number} for convex optimization introduced in the seminal papers~\cite{Rene95a,Rene95b} that has been further extended in ~\cite{AmelB12,BurgC13,EpelF02,FreuV99a,FreuV03,Freu04,Pena00,Pena03} among  many other articles.  

\medskip

The identities~\eqref{eq.HA.J} and~\eqref{eq.HA} readily yield the following bound on $H(A)$ previously established in~\cite{burke1996,guler1995}
\begin{align*}
H(A) &= \max_{J \in \sJ(A)} \max\{\|v\|^*: v \in\R^J_+, \|A_J\transp v\|^* \le 1\}\\
&= \max_{J \in \sJ(A)} \max\{\|\tilde v\|^*: \tilde v \in \ext\{v \in\R^J_+, \|A\transp v\|^* \le 1\}\}
\\
&\le \max\{\|\tilde v\|^*: \tilde v \in \ext\{v \in\R^m_+, \|A\transp v\|^* \le 1\}\}.
\end{align*}
In the above expressions $\ext(C)$ denotes the set of extreme points of a closed convex set $C$.

\medskip


Let $A\in \R^{m\times n}.$  Observe that if $J \subseteq F \subseteq\{1,\dots,m\}$ and $F$ is $A$-surjective then $J$ is $A$-surjective.  In other words, if $J \subseteq F \in \sJ(A)$ then $F$ provides a {\em certificate of surjectivity} for $J$.  Equivalently, if $I \subseteq J$ and $I$ is not $A$-surjective then $J$ is not $A$-surjective, that is, $I$ provides a {\em certificate of non-surjectivity} for $J$.  The following corollary of Proposition~\ref{prop.Hoffman.gral} takes this observation a bit further and provides the crux of our combinatorial algorithm to compute $H(A)$.

\begin{corollary}\label{corol.sets} Let $A \in \R^{m\times n}$.  Suppose $\F\subseteq \sJ(A)$ and $\I \subseteq 2^{\{1,\dots,m\}}\setminus \sJ(A)$  provide joint certificates of surjectivity and non-surjectivity for all subsets of $\{1,\dots,m\}$.  In other words, for all $J \subseteq \{1,\dots,m\}$ either  $J\subseteq F$ for some $F \in \F,$
or $I\subseteq J$ for some $I \in \I.$   Then 
\[
H(A) =  \dmax_{F \in \F} H_F(A).
\]
\end{corollary}
\begin{proof}
The conditions on $\F$ and $\I$ imply that for all $J \in \sJ(A)$ there exists $F \in \F$ such that $J\subseteq F$.  The latter condition implies that $H_J(A) \le H_F(A)$. Therefore Proposition~\ref{prop.Hoffman.gral} yields
$$
H(A) = \dmax_{J \in \sJ(A)} H_J(A) =  \dmax_{F \in \F} H_F(A).
$$
\end{proof}

\medskip

Proposition~\ref{prop.Hoffman.gral}, Proposition~\ref{prop.Hoffman.A.surj}, and Corollary~\ref{corol.sets}  extend to the more general context when some of the inequalities in $Ax \le b$ are easy to satisfy.  This occurs in particular when some of the inequalities $Ax \le b$ are of the form $ x \le u$ or $ -x \le -\ell$.    It is thus  natural to consider a refinement of the Hoffman constant $H(A)$ that reflects the presence of this kind of easy-to-satisfy constraints.  

Let $A\in \R^{m\times n}$ and $L\subseteq \{1,\dots,m\}$.  Let $L^c:= \{1,\dots,m\}\setminus L$ denote the complementary set of $L$.
Define 
\begin{equation}\label{eq.Hoffman.matr.rest}
H(A\vert L):=\dmax_{J\in \sJ(A)} H_J(A\vert L)
\end{equation}
where
\[
H_J(A\vert L):= \dmax_{y\in \R^L \atop \|y\|\le 1} \dmin_{x\in \R^n \atop A_Jx \le y_J} \|x\|
\]
for each $J\in \sJ(A)$.  For ease of notation, the latter expression uses the convention that $y_j = 0$ whenever $j \in J\setminus L$.  In particular, observe that $H_J(A\vert L) = 0$ if $J\cap L = \emptyset.$  For $b\in \R^m$ and $S\subseteq \R^m$ let
\[
\dist_L(b,S):= \inf\{\|b-y\|: y\in S, (b-y)_{L^c} = 0\}.
\]
Evidently $\dist_L(b,S) < \infty$ if and only if $(S-b) \cap \{y \in \R^m: y_{L^c} = 0\}\ne \emptyset$.

Proposition~\ref{prop.Hoffman.gral}, Proposition~\ref{prop.Hoffman.A.surj}, and Corollary~\ref{corol.sets} extend to a system of inequalities of the form $Ax \le b$ where the subset of inequalities $A_{L^c} x \le b_{L^c}$ is easy to satisfy. 
 
\begin{proposition}\label{prop.Hoffman.gral.rest}  
Let  $A\in \R^{m\times n}$ and $L\subseteq\{1,\dots,m\}.$  Then for all $b \in \R^m$ such that $P_{A,b} := \{x\in\R^n: Ax\le b\}\ne \emptyset$ and all $u\in \{x \in \R^n: A_{L^c} x \le b_{L^c}\}$
\begin{equation}\label{eq.Hoffman.bound.matr}
\dist(u,P_{A,b}) \le H(A\vert L)\cdot \dist_L(b,Au + \R^m_+) \le H(A\vert L)\cdot\|(A_Lu-b_L)_+\|.
\end{equation}
Furthermore, the first bound is tight: If $H(A\vert L) >0$ then there exist $b\in \R^m$ 
such that $P_{A,b} \ne \emptyset$ and $u \in 
 \{x \in \R^n: A_{L^c} x \le b_{L^c}\} \setminus P_{A,b}$ such that 
\[
\dist(u,P_{A,b}) = H(A\vert L)\cdot \dist_L(b,Au + \R^m_+).
\]
\end{proposition}

\begin{proposition}\label{prop.Hoffman.A.surj.rest} 
Let  $A\in \R^{m\times n}\setminus\{0\}$ and $L\subseteq\{1,\dots,m\}$.  Then  for all $J\in \sJ(A)$
\begin{equation}\label{eq.HA.J.L}
H_J(A\vert L) = 
\dmax_{y\in \R^L \atop \|y\|\le 1} \dmin_{x\in \R^n\atop A_J x \le y_J} \|x\| = 
\dmax_{v \in \R^J_+ \atop \|A_J\transp v\|^*\le 1} \|v_{J\cap L}\|^* =
\frac{1}{\dmin_{v \in \R^{J}_+, \; \|v_{J\cap L}\|^* =1 } \|A_{J}\transp v\|^*}
\end{equation}
with the convention that the denominator in the last expression is $+\infty$ when $J\cap L = \emptyset.$ 

If the mapping $x\mapsto Ax+\R^m_+$ is surjective then
\begin{equation}\label{eq.HA.L}
H(A\vert L) = \dmax_{y\in \R^L \atop \|y\|\le 1} \dmin_{x\in \R^n \atop A_L x \le y} \|x\| = \dmax_{v \in \R^m_+ \atop \|A\transp v\|^*\le 1} \|v_L\|^* = 
\frac{1}{\dmin_{v \in \R^m_+, \; \|v_L\|^* =1 } \|A\transp v\|^*}.
\end{equation}
\end{proposition}

\begin{corollary}\label{corol.sets.rest} Let $A \in \R^{m\times n}$ and $L\subseteq \{1,\dots,m\}$.  Suppose $\F\subseteq \sJ(A)$ and $\I \subseteq 2^{\{1,\dots,m\}}\setminus \sJ(A)$  are such that for all $J \subseteq \{1,\dots,m\}$ either  $J\subseteq F$ for some $F \in \F,$
or $I\subseteq J$ for some $I \in \I.$   Then 
\[
H(A\vert L) =  \dmax_{F \in \F} H_F(A\vert L).
\]
\end{corollary}

\subsection{The case of equations and inequalities}

The previous statements extend to linear systems of equations and inequalities combined.  Proposition~\ref{prop.Hoffman} below gives a bound analogous to~\eqref{eq:simple_erro} for the distance from a point $u\in \R^n$ to a nonempty polyhedron of the form 
\[
\{x\in \R^n: Ax = b, \, Cx \le d\} = A^{-1}(b) \cap P_{C,d}.
\]
Here and throughout this section $A^{-1}:\R^m \rightrightarrows \R^n$ denotes the inverse set-valued mapping of the linear mapping $x\mapsto Ax$ defined by a matrix $A \in \R^{m\times n}$.

\medskip

Let $A \in \R^{m\times n}, \, C \in \R^{p \times n}$.  For $J\subseteq\{1,\dots,p\}$ let $[A,C,J]:\R^n \rightrightarrows \R^m \times \R^p$ be the set-valued mapping defined by
\[
x \mapsto \matr{Ax\\Cx} + \left\{\matr{0\\s}: s\in \R^p, \, s_J \ge 0 \right\}.
\]
Define
\[
\sJ(A;C):=\{J \subseteq \{1,\dots,p\}: [A,C,J] \text{ is relatively surjective}\}.
\]

\begin{proposition}\label{prop.Hoffman} Let $A \in \R^{m\times n}, \, C \in \R^{p \times n}$, and $H:=\dmax_{J\in \sJ(A;C)} H_J$ where
\begin{equation}\label{eq.Hoffman}
H_J:=
\dmax_{(y,w)\in (A\R^n)\times \R^p \atop \|(y,w)\|\le 1} \dmin_{x\in \R^n \atop Ax = y, C_Jx \le w_J} \|x\| = \dmax_{(v,z)\in (A\R^n)\times\R^p_+ \atop z_{J^c} = 0,\|A\transp v + C\transp z\|^* \le 1} \|(v,z)\|^*
= \frac{1}{\dmin_{v\in A\R^n, z\in \R^p_+ \atop z_{J^c} = 0,\|(v,z)\|^* = 1} 
\|A\transp v + C\transp z\|^*}.
\end{equation}
Then for all $b\in \R^m, d\in \R^p$ such that $A^{-1}(b) \cap P_{C,d} := \{x\in \R^n: Ax = b, \, Cx \le d\}\ne \emptyset$ and all $u\in \R^n$
\[
\dist(u,A^{-1}(b)\cap P_{C,d}) \le H\cdot \dist\left(\matr{b\\d},\matr{Au\\Cu} + \{0\}\times\R^p_+\right) \le H\cdot\left\|\matr{Au-b\\(Cu-d)_+}\right\|.
\]
The first bound is tight: If $H>0$ then there exist $b\in \R^m, \, d\in\R^p$ such that $A^{-1}(b)\cap P_{C,d} \ne \emptyset$ and $u \not \in A^{-1}(b)\cap P_{C,d}$ such that 
\[
\dist(u,A^{-1}(b)\cap P_{C,d}) =  H\cdot \dist\left(\matr{b\\d},\matr{Au\\Cu} + \{0\}\times\R^p_+\right).
\]
If $[A,C,\{1,\dots,p\}] : \R^n \rightrightarrows \R^m \times \R^p$ 
 is  surjective then
\begin{equation}\label{eq.Hoffman.surj}
H = \dmax_{v\in \R^m, z\in \R^p_+ \atop \|A\transp v + C\transp z\|^* \le 1} \|(v,z)\|^*
= \frac{1}{\dmin_{v\in \R^m, z\in \R^p_+ \atop \|(v,z)\|^* = 1} 
\|A\transp v + C\transp z\|^*}.
\end{equation}
\end{proposition}

\medskip

Proposition~\ref{prop.Hoffman} also extends to the case when some of the equations or inequalities in $Ax = b, \; Cx \le d$ are easy to satisfy. 
We next detail several special but particularly interesting cases.
The next proposition considers systems of equations and inequalities when the inequalities are easy.  This case plays a central role in~\cite{Garb18}.

\begin{proposition}
\label{prop.Hoffman.std} 
Let $A \in \R^{m\times n}, \, C \in \R^{p \times n},$ and $H:=\dmax_{J\in \sJ(A;C)} H_J$ where
\begin{equation}\label{eq.Hoffman.std}
H_J:= 
\dmax_{y\in \{Ax: C_Jx\le 0\}\atop \|y\|\le 1} \dmin_{x\in \R^n \atop Ax = y, C_Jx \le 0} \|x\|
= \dmax_{(v,z)\in (A\R^n)\times\R^p_+ \atop z_{J^c} = 0, \|A\transp v + C\transp z\|^* \le 1} \|v\|^*
= \frac{1}{\dmin_{(v,z)\in (A\R^n)\times\R^p_+ \atop z_{J^c} = 0,\|v\|^* = 1} 
\|A\transp v + C\transp z\|^*}.
\end{equation}
Then for all $b\in \R^m, d\in \R^p$ such that $A^{-1}(b) \cap P_{C,d}\ne \emptyset$ and all $u\in P_{C,d}$
\[
\dist(u,A^{-1}(b)\cap P_{C,d}) \le  H\cdot\|Au-b\|.
\]
This bound is tight: If $H>0$ then there exist $b\in \R^m, \, d\in\R^p$ such that $A^{-1}(b)\cap P_{C,d} \ne \emptyset$ and $u  \in
 P_{C,d} \setminus A^{-1}(b)$ such that 
\[
\dist(u,A^{-1}(b)\cap P_{C,d}) =  H \cdot \|Au-b\|.
\]
If $[A,C,\{1,\dots,p\}] : \R^n \rightrightarrows \R^m \times \R^p$ 
 is  surjective then
\begin{equation}\label{eq.Hoffman.std.surj}
H=\dmax_{(v,z)\in \R^m\times \R^p_+ \atop \|A\transp v + C\transp z\|^* \le 1} \|v\|^*
= \frac{1}{\dmin_{(v,z)\in \R^m\times \R^p_+ \atop \|v\|^* = 1} 
\|A\transp v + C\transp z\|^*}.
\end{equation}

\end{proposition}

Notice the analogy between Proposition~\ref{prop.Hoffman.std} and the following classical error bound for systems of linear equations. Let $A \in \R^{m\times n}$ be full row rank.  Then for all $b \in \R^m$ and $u\in \R^n$
\[
\dist(u,A^{-1}(b)) \le \|A^{-1}\| \cdot \|Au - b\|
\]
where
\[
\|A^{-1}\| = \dmax_{y\in\R^m\atop \|y\|\le 1} \min_{x\in A^{-1}(y)} \|x\|=\dmax_{v \in \R^m \atop \|A\transp v\|^*\le 1} \|v\|^* =  \frac{1}{
\dmin_{v \in \R^m \atop \|v\|^* = 1} \|A\transp v\|^*
}
\]
is the norm of the inverse mapping $A^{-1}:\R^m \rightrightarrows \R^n$ defined by $A$.

\bigskip

Next, consider the case when the equations are easy.  This case plays a central role in~\cite{xia2015}.

\begin{proposition}\label{prop.equal.easy}
Let $A \in \R^{m\times n}, \, C \in \R^{p \times n}$, and $H:=\dmax_{J\in \sJ(A;C)} H_J$ where
\begin{equation}\label{eq.equal.easy}
H_J := 
\dmax_{w\in \R^p\atop \|w\|\le 1} \dmin_{x\in \R^n \atop Ax = 0, C_Jx \le w_J} \|x\|
= \dmax_{(v,z)\in (A\R^n)\times\R^p_+ \atop z_{J^c} = 0, \|A\transp v + C\transp z\|^* \le 1} \|z\|^*
= \frac{1}{\dmin_{(v,z)\in (A\R^n)\times\R^p_+ \atop z_{J^c} = 0, \|z\|^* = 1} 
\|A\transp v + C\transp z\|^*}.
\end{equation}
Then for all $b\in \R^m, d\in \R^p$ such that $A^{-1}(b) \cap P_{C,d}\ne \emptyset$ and all $u\in A^{-1}(b)$
\[
\dist(u,A^{-1}(b) \cap P_{C,d}) \le H \cdot \dist\left(d,Cu + \R^p_+\right) \le H\cdot\|(Cu-d)_+\|.
\]
The first bound is tight: If $H>0$ then there exist $b\in \R^m, \, d\in\R^p$ such that $A^{-1}(b) \cap P_{C,d} \ne \emptyset$ and $u  \in
P_{C,d} \setminus A^{-1}(b)$ such that 
\[
\dist(u,A^{-1}(b) \cap P_{C,d}) =  H \cdot \dist\left(d,Cu + \R^p_+\right).
\]
If $[A,C,\{1,\dots,p\}] : \R^n \rightrightarrows \R^m \times \R^p$  is  surjective then 
\begin{equation}\label{eq.equal.easy.surj}
H= 
\dmax_{(v,z)\in \R^m\times \R^p_+ \atop \|A\transp v + C\transp z\|^* \le 1} \|z\|^* =
\frac{1}{\dmin_{(v,z)\in \R^m\times \R^p_+ \atop \|z\|^* = 1} \|A\transp v + C\transp z\|^*}.
\end{equation}
\end{proposition}

\bigskip

When the mapping $[A,C,\{1,\dots,p\}]$ is surjective, the quantity $1/H$
defined in each of Proposition~\ref{prop.Hoffman}, Proposition~\ref{prop.Hoffman.std}, or Proposition~\ref{prop.equal.easy} equals a certain kind of {\em block-structured distance to non-surjectivity} of $[A,C,\{1,\dots,p\}]$.  More precisely, when $[A,C,\{1,\dots,p\}]$ is surjective, the quantity $1/H$ defined by~\eqref{eq.Hoffman.surj} equals the size of 
the smallest $ (\Delta A,\Delta C)\in \R^{(m+p)\times n}$ such that $[A+\Delta A,C+\Delta C,\{1,\dots,p\}]$ is not surjective.  
Similarly, when $[A,C,\{1,\dots,p\}]$ is surjective, the quantity $1/H$ defined by~\eqref{eq.Hoffman.std.surj} equals the size of 
the smallest $\Delta A \in \R^{m\times n}$ such that $[A+\Delta A, C,\{1,\dots,p\}]$ is not surjective.  Finally, when $[A,C,\{1,\dots,p\}]$ is surjective, the quantity $1/H$ defined by~\eqref{eq.equal.easy.surj} equals the size of 
the smallest $\Delta C \in \R^{{p}\times n}$ such that $[A,C+\Delta C,\{1,\dots,p\}]$ is not surjective. Each of these block-structured distances to non-surjectivity is the same as the analogous block-structured distances to ill-posedness of the system of inequalities $Ax = 0, \, Cx < 0$.  For a more detailed discussion on the block-structure distance to non-surjectivity and the block-structure distance to ill-posedness, we refer the reader to~\cite{Lewi05,Pena00,Pena03,Pena05}.

\bigskip

Next, consider a special case when one of the equations and all inequalities are easy.  This case underlies the construction of some measures of conditioning for polytopes developed in~\cite{BeckS15,GarbH13,LacoJ15,GutmP18,PenaR16} to establish the linear convergence of some variants of the Frank-Wolfe Algorithm.   Recall some notation from~\cite{GutmP18}. Let $A\in \R^{m\times n}$ and consider the polytope
$
\conv(A) := A\Delta_{n-1},
$
where $\Delta_{n-1} := \{x\in \R^n_+: \|x\|_1 = 1\}$.
Observe that $v \in \conv(A)$ if and only if the following system of constraints has a solution
\begin{equation}\label{eq.polytope}
 Ax = v, \; x \in \Delta_{n-1}.
 \end{equation}
It is natural to consider $x \in \Delta_{n-1} 
$ 
in~\eqref{eq.polytope} as an {\em easy-to-satisfy} constraint.  
Following the notation in~\cite{GutmP18}, for $v\in \conv(A)$ let \[Z(v) := \{z\in \Delta_{n-1}: Az = v\}.\]  The Hoffman constant $H$ in Proposition~\ref{prop.facial.dist} below plays a central role in~\cite{GutmP18,LacoJ15,PenaR16}.  In particular, when $\R^n$ is endowed with the $\ell_1$ norm, $1/H$ is the same as the {\em facial distance} or {\em pyramidal width} of the polytope $\conv(A)$ as detailed in~\cite{GutmP18,PenaR16}.  We will rely on the following notation.  For $A\in \R^{m\times n}$ let $L_A:=\{Ax: \1\transp x = 0\}$ and  for $J\subseteq \{1,\dots,n\}$  let $K_J:= \{x\in \R^n: \1\transp x = 0, \; x_J \ge 0\}$.  Let $\Pi_{L_A}:\R^m \rightarrow L_A$ denote the orthogonal projection onto $L_A$.

\begin{proposition}\label{prop.facial.dist} Let $A\in \R^{m\times n}$.  Let $\tilde A := \matr{A \\ \1\transp} \in \R^{(m+1)\times n},\; C:=-I_n \in \R^{n\times n},$ and $H:=\dmax_{J\in \sJ(\tilde A;C)} H_J$ where
\begin{equation}\label{eq.facial.dist}
H_J:= 
 \max_{y \in A K_J\atop \|y\| \le 1} \dmin_{x\in K_J \atop Ax = y} \|x\| = \max_{(v,t) \in \tilde A\R^n,z\in\R^n_+ \atop z_{J^c} = 0, \|A\transp v + t\1- z\|^* \le 1} \|\Pi_{L_A}(v)\|^* = \frac{1}{\dmin_{(v,t) \in \tilde A\R^n, z\in  \R^n_+ \atop z_{J^c} = 0, \|\Pi_{L_A}(v)\|^* = 1} \|A\transp v + t\1 - z\|^*}.
\end{equation}
Then for all $x \in \Delta_{n-1}$ and $v\in \conv(A)$ 
\[
\dist(x,Z(v)) \le H \cdot \|Ax - v\|.
\]
Furthermore, this bound is tight: If $H>0$ then there exist $v\in \conv(A)$ and $x\in \Delta_{n-1}\setminus Z(v)$ such that
\[
\dist(x,Z(v)) = H \cdot \|Ax - v\| > 0.
\]
\end{proposition}

 \medskip

The following analogue of Corollary~\ref{corol.sets} and Corollary~\ref{corol.sets.rest} also holds.

\begin{corollary}\label{corol.sets.gral} Let $A \in \R^{m\times n},\, C\in \R^{p\times n}$.  Suppose $\F\subseteq \sJ(A;C)$ and $\I \subseteq 2^{\{1,\dots,p\}}\setminus \sJ(A;C)$  are such that for all $J \subseteq \{1,\dots,p\}$ either  $J\subseteq F$ for some $F \in \F,$
or $I\subseteq J$ for some $I \in \I.$   Then the expression 
$H:=\max_{J\in \sJ(A;C)} H_J$ in each of Proposition~\ref{prop.Hoffman}, Proposition~\ref{prop.Hoffman.std}, and Proposition~\ref{prop.equal.easy} can be replaced with $H=\max_{F\in \F} H_F$.  The same holds for Proposition~\ref{prop.facial.dist} with $\tilde A = \matr{A \\ \1\transp}$ in lieu of $A$.
\end{corollary}

\section{An algorithm to compute the Hoffman constant}
\label{sec.algo}

We next describe an algorithm to compute the Hoffman constant of a systems of linear equations and inequalities.  We first describe the computation of the Hoffman constant $H(A)$ in Proposition~\ref{prop.Hoffman.gral}.  We subsequently describe the computation of the Hoffman constant $H(A;C) := H$ 
defined in Proposition~\ref{prop.Hoffman}.  

The algorithms described below have straightforward extensions to the more general case when some equations or inequalities are easy to satisfy.

\subsection{Computation of $H(A)$}
\label{sec.algo.ineq}
Let $A\in \R^{m\times n}.$  Corollary~\ref{corol.sets} suggests the following algorithmic approach to compute $H(A)$:  Find collections of sets $\F \subseteq \sJ(A)$ and $\I \subseteq 2^{\{1,\dots,m\}} \setminus \sJ(A)$ that  provide joint certificates of surjectivity and non-surjectivity for all subsets of $\{1,\dots,m\}$ and then compute $H(A) = \max_{F\in \F} H_F(A)$.   A naive way to construct $\F$ and $\I$ would be to scan the subsets of $\{1,\dots,m\}$ in monotonically decreasing order as follows.   Starting with $J = \{1,\dots,m\}$, check whether $J$ is surjective.
If $J$ is surjective, then place $J$ in $\F$.  Otherwise, place $J$ in $\I$ and continue by scanning each $J\setminus\{i\}$ for $i\in J$. 
Algorithm~\ref{alg:bb} and its variant, Algorithm~\ref{alg:bb.v2}, refine the above naive approach to construct $\F,\I$ 
more efficiently.  We next describe both algorithms.

The central idea of Algorithm~\ref{alg:bb} is to maintain three collections $\F,\I,\J \subseteq 2^{\{1,\dots,m\}}$ such that the following invariant holds at the beginning of each main iteration (Step 3 in Algorithm~\ref{alg:bb}):

\begin{quote}
The collections $\F \subseteq \sJ(A)$ and $\I \subseteq 2^{\{1,\dots,m\}}\setminus \sJ(A)$  provide joint certificates of surjectivity and non-surjectivity for all subsets of $\{1,\dots,m\}$ except possibly those included in some subset in the collection $\J$.
\end{quote}

This invariant evidently holds for $\F = \I = \emptyset$ and $\J=\{\{1,\dots,m\}\}.$  At each main iteration, Algorithm~\ref{alg:bb} scans a set $J\in \J$ to either detect that $J$ is $A$-surjective or find a certificate of non-surjectivity $I\subseteq J$.  If $J$ is $A$-surjective then the above invariant continues to hold after adding $J$ 
 to $\F$ and removing all $\tilde J\in \J$ such that $\tilde J \subseteq J$.  On the other hand, if $I$ is a certificate of  non-surjectivity for $J$, then the invariant continues to hold if $I$ is added to $\I$ and $\J$ is updated as follows. Replace each $\hat J\in \J$ that contains $I$ with the sets $\hat J\setminus \{i\}, \; i\in I$ that are not included in any set in $\F$.  Algorithm~\ref{alg:bb} terminates when $\J$ is empty.  This must happen eventually since at each main iteration the algorithm either removes at least one subset from $\J$ or removes at least one subset from $\J$ and replaces it by proper subsets of it.  

\medskip

The most time-consuming operation in Algorithm~\ref{alg:bb} (Step 4) is the step that detects whether a subset $J\in \J$ is $A$-surjective or finds a certificate of nonsurjectivity $I\subseteq J$.  This step requires solving the following problem
\begin{equation}\label{eq.lp}
\min\{\|A_J\transp v\|^*: v\in \R^J_+, \|v\|^* = 1\}.
\end{equation}
Observe that $J$ is $A$-surjective if and only if the optimal value of~\eqref{eq.lp} is positive. More precisely,
by Proposition~\ref{prop.Hoffman.A.surj}, the minimization problem~\eqref{eq.lp} either detects that $J$ is $A$-surjective and computes 
$1/H_J(A)$ when its optimal value is positive, or detects that $J$ is not $A$-surjective and finds $v\in \R^J_+\setminus\{0\}$ such that $A_J\transp v = 0$.  In the latter case, the set $I(v):=\{i\in J: v_i > 0\}$ is a certificate of non-surjectivity for $J$.  When $J$ is not $A$-surjective, the certificate of non-surjectivity $I(v)\subseteq J$ obtained from~\eqref{eq.lp} is typically smaller than $J$.

The tractability of problem~\eqref{eq.lp} depends on the norms in $\R^n$ and $\R^m$.  In particular,
when~$\R^m$ is endowed with the $\ell_\infty$-norm we have $\|v\|^* = \|v\|_1 = \1\transp v$ for $v\in \R^J_+$ and thus~\eqref{eq.lp} becomes the following convex optimization problem 
\[
\min\{\|A_J\transp v\|^*: v\in \R^J_+, \1\transp v = 1\}.
\]
Furthermore,~\eqref{eq.lp} is a linear program if both $\R^m$ and $\R^n$ are endowed with the $\ell_\infty$-norm or if~$\R^m$ is endowed with the $\ell_\infty$-norm and $\R^n$ is endowed with the $\ell_1$-norm.  Problem~\eqref{eq.lp} is a second-order conic program if $\R^m$ is endowed with the $\ell_\infty$-norm and $\R^n$ is endowed with the $\ell_2$-norm.  In our MATLAB prototype implementation described below,  $\R^n$ and $\R^m$  
are endowed with the~$\ell_\infty$ norm and~\eqref{eq.lp} is solved via linear programming.  

Problem~\eqref{eq.lp} can also be solved by solving $|J|$ convex optimization problems when $\R^m$ is endowed with the $\ell_1$-norm.  This is suggested by the characterizations of Renegar's distance to ill-posedness in~\cite{FreuV99a,FreuV03}.  When $\R^m$ is endowed with the $\ell_1$-norm we have $\|v\|^* = \|v\|_\infty = \dmax_{j\in J} v_j$ for $v\in \R^J_+$ and thus
\[
\min\{\|A_J\transp v\|^*: v\in \R^J_+, \|v\|^* = 1\} = \dmin_{j\in J} \;  \min\{\|A_J\transp v\|^*: v\in \R^J_+, v \le \1, \, v_j = 1\}.
\] 
Section~\ref{sec.euclidean} below describes a more involved approach to estimate~\eqref{eq.lp} when both $\R^n$ and $\R^m$  
are endowed with the $\ell_2$ norm.

We should note that although the specific value of the Hoffman constant $H(A)$ evidently depends on the norms in $\R^n$ and $\R^m$, the $A$-surjectivity of a subset $J\subseteq\{1,\dots,m\}$ does not.  In particular, the collections $\F,\I$ found in  Algorithm~\ref{alg:bb} could be used to compute or estimate $H(A)$ for any arbitrary norms provided each $H_F(A)$ can be computed or estimated when $F\subseteq\{1,\dots,m\}$ is $A$-surjective.

A potential drawback of Algorithm~\ref{alg:bb} is the size of the collection $\J$ that could become potentially large even if the sets $\F,\I$ do not.  This drawback suggests an alternate approach.  Given $\F\subseteq \sJ(A)$ and $\I \subseteq 2^{\{1,\dots,m\}}\setminus \sJ(A)$ consider the feasibility problem
\begin{equation}\label{eq.ip}
\begin{array}{rl}
& |J^c\cap I| \ge 1, \; I\in \I\\
& |J\cap F^c| \ge 1, \; F\in \F\\
& J \subseteq \{1,\dots,m\}.
\end{array}
\end{equation}
Observe that $\F,\I$ jointly provide certificates of surjectivity or non-surjectivity for all subsets of $\{1,\dots,m\}$ if and only if~\eqref{eq.ip} is infeasible.  This suggests the variant of Algorithm~\ref{alg:bb} described in Algorithm~\ref{alg:bb.v2}.  The main difference is that  Algorithm~\ref{alg:bb.v2} does not maintan $\J$ and instead relies on~\eqref{eq.ip} at each main iteration.  Algorithm~\ref{alg:bb.v2} trades off the memory cost of maintaining $\J$ for the computational cost of solving the feasibility problem~\eqref{eq.ip} at each main iteration.

\begin{algorithm}
  \caption{
   Computation of collections of certificates $\F ,\;\I$ and constant $H(A)$}
\label{alg:bb}
\begin{algorithmic}[1]
\State {\bf input} $A \in \R^{m \times n}$
\State Let $\F := \emptyset, \;\I := \emptyset, \; \J:=\{\{1,\dots,m\}\}, H(A):= 0$
\While {$\mathcal J \ne \emptyset$}
	\State Pick $J \in \J$ 
	 and let $v$ solve~\eqref{eq.lp} to detect whether $J$ is $A$-surjective
		\If {$\|A_J\transp v\|^* > 0$}
			\State 
			$\F := \F \cup \{J\}$, $\hat \J := \{\hat J \in \J: \hat J \subseteq J\},$  and  $H(A) := \max\left\{H(A),\frac{1}{\|A_J\transp v\|^*}\right\}$
			\State Let $\J:=\J\setminus\hat \J$ 
		\Else
			\State 
			 Let $\I:= \I \cup\{I(v)\}$, $\hat \J := \left\{\hat J\in \J: I(v) \subseteq \hat J\right\}$
			 \State Let $\bar \J := \left\{\hat J\setminus \{i\}: \hat J \in \hat \J, i\in I(v), \hat J\setminus \{i\} \not \subseteq F \text{ for all } F\in \F\right\}$ 
			 \State Let $\J:= (\J \setminus \hat \J) \cup \bar \J$		\EndIf
		\EndWhile
\State \Return $\F, \, \I, \, H(A)$
\end{algorithmic}
\end{algorithm}

\begin{algorithm}
  \caption{
   Computation of collections of certificates $\F ,\;\I$ and constant $H(A)$ version 2}
\label{alg:bb.v2}
\begin{algorithmic}[1]
\State {\bf input} $A \in \R^{m \times n}$
\State Let $\F := \emptyset, \;\I := \emptyset, H(A):= 0$
\While {\eqref{eq.ip} is feasible}
	\State Let $J\subseteq \{1,\dots,m\}$ solve  \eqref{eq.ip} and let $v$ solve~\eqref{eq.lp} to detect whether $J$ is $A$-surjective
		\If {$\|A_J\transp v\|^* > 0$}
			\State 
			$\F := \F \cup \{J\}$ and  $H(A) := \max\left\{H(A),\frac{1}{\|A_J\transp v\|^*}\right\}$
		\Else
			\State 
			 Let $\I:= \I \cup\{I(v)\}$	
			 \EndIf
		\EndWhile
\State \Return $\F, \, \I, \, H(A)$
\end{algorithmic}
\end{algorithm}

We tested prototype MATLAB implementations of Algorithm~\ref{alg:bb} and 
Algorithm~\ref{alg:bb.v2} on collections of randomly generated matrices $A$ of various sizes (with $m >n$, where the analysis is interesting).  The entries in each matrix were drawn from independent standard normal distributions.  Figure~\ref{the.figure} summarizes our results.  It displays boxplots for the sizes of the sets $\F, \I$ at termination for the {\em non-surjective} instances in the sample, that is, the matrices~$A$ such that $0\in \conv(A\transp)$.  We excluded the {\em surjective} instances, that is, the ones with $0\not\in \conv(A\transp)$ because for those instances the collections $\F = \{\{1,\dots,m\}\}$ and $\I = \emptyset$ provide certificates of surjectivity for all subsets of $\{1,\dots,m\}$ and are  identified at the first iteration of the algorithm when $J=\{1,\dots,m\}$ is scanned.  Thus the non-surjective instances are the interesting ones.  
As a reality check to our implementation of both algorithms, for every instance that we tested, we used~\eqref{eq.ip} to verify that the final sets $\F, \I$ indeed provide certificates of surjectivity and non-surjectivity for all subsets of $\{1,\dots,m\}.$

\begin{figure}[!htb]
\begin{center}
\includegraphics[width=\textwidth]{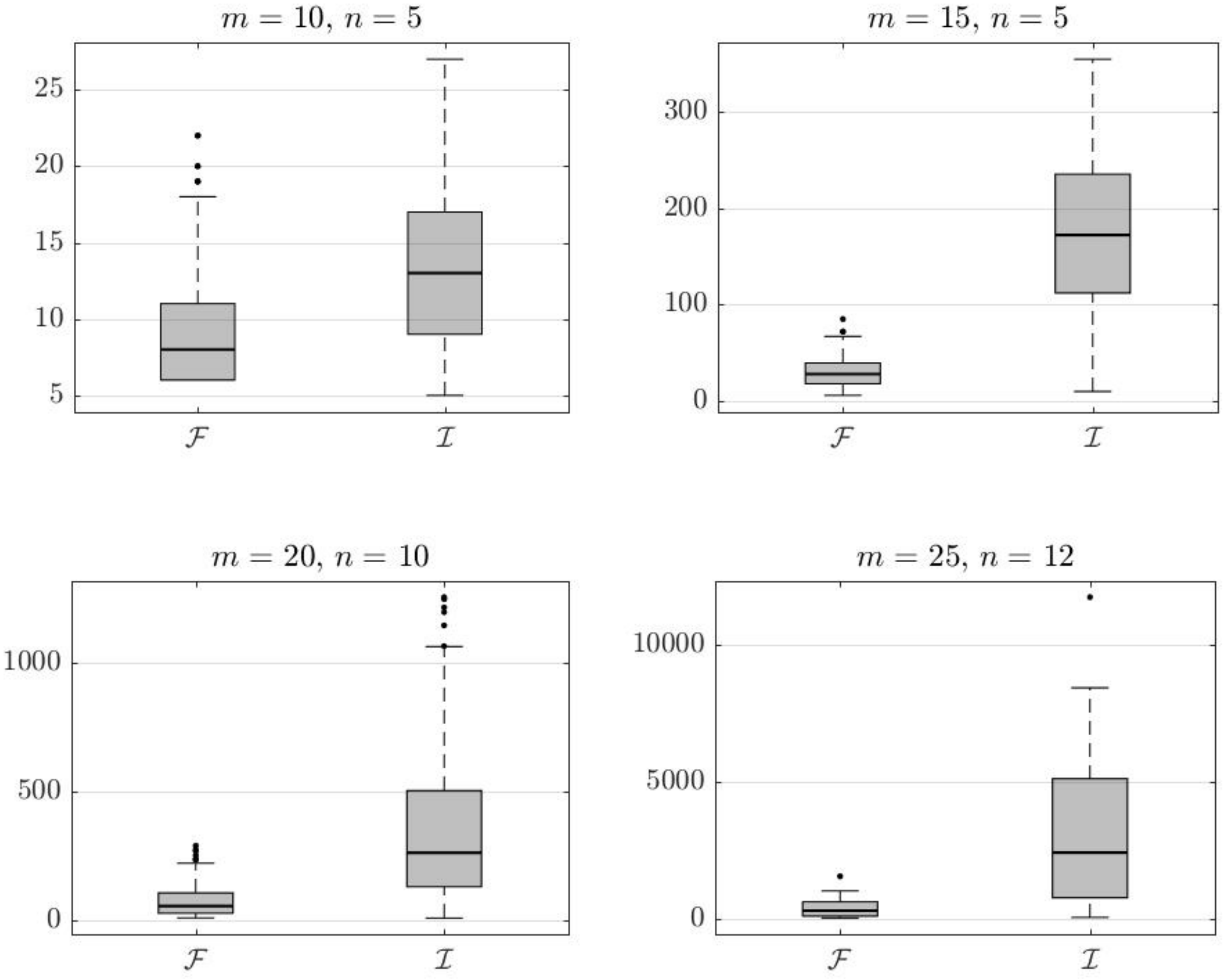}
\end{center}
\caption{Box plots of the distributions of the sizes of the sets $\I$ and $\F$ for the non-surjective instances obtained after randomly sampling 1000 matrices with $m$ rows and $n$ columns.}
\label{the.figure}
\end{figure}

The MATLAB code and scripts used for our experiments are publicly available in the following website

\begin{center} {\tt http://www.andrew.cmu.edu/user/jfp/hoffman.html} \end{center} 

The reader can readily use these files to replicate numerical results similar to those summarized in Figure~\ref{the.figure}.

 It is interesting to note that
 the size of the collection $\F$ in our experiments does not grow too rapidly.  This is reassuring in light of the characterization $H(A) = \max_{F\in \F} H_F(A)$.   Our prototype implementations are fairly basic.  In particular, our prototype implementation of Algorithm~\ref{alg:bb} maintains an explicit representation of the collections $\F, \I, \J$.  Our prototype implementation of Algorithm~\ref{alg:bb.v2} solves~\eqref{eq.ip} via integer programming.  Neither of them use warm-starts.  
It is evident that the collections $\F, \I, \J$ as well as the feasibility problem~\eqref{eq.ip} could all be handled more efficiently via more elaborate combinatorial structures such as binary decision diagrams~\cite{Aker78,BergCVHH16}.  A clever use of warm-starts would likely boost efficiency since the algorithms need to solve many similar linear and integer programs.

The results of the prototype implementations of Algorithm~\ref{alg:bb} and Algorithm~\ref{alg:bb.v2} are encouraging and suggest that  more sophisticated implementations could compute the Hoffman constant $H(A)$ for much larger matrices. 

\subsection{Computation of $H(A;C)$}
Throughout this subsection we let $H(A;C)$ denote the Hoffman constant $H$ defined in Proposition~\ref{prop.Hoffman}.
Let $A\in \R^{m\times n}$ and $C\in \R^{p\times n}$.  In parallel to the  observation in Section~\ref{sec.algo.ineq} above, Corollary~\ref{corol.sets.gral} suggests the following approach to compute $H(A;C)$:  Find collections of sets $\F \subseteq \sJ(A;C)$ and $\I \subseteq 2^{\{1,\dots,p\}} \setminus \sJ(A;C)$ such that for all $J\subseteq \{1,\dots,p\}$ either  $J\subseteq F$ for some $F \in \F,$
or $I\subseteq J$ for some $I \in \I$.  Then compute $H(A;C):=\dmax_{J\in \F} H_J(A;C)$ where 
\[
H_J(A;C) =  \frac{1}{\dmin_{v\in A\R^n, z\in \R^J_+ \atop \|(v,z)\|^* = 1} 
\|A\transp v + C_J\transp z\|^*}.
\] 
Algorithm~\ref{alg:bb} has the straightforward extension described in Algorithm~\ref{alg:bb.gral} to find $\F \subseteq \sJ(A;C)$ and $\I \subseteq 2^{\{1,\dots,p\}} \setminus \sJ(A;C)$ as above.  Algorithm~\ref{alg:bb.v2} has a similar straightforward extension.  
The most time-consuming operation in Algorithm~\ref{alg:bb.gral} (Step 4) is the step that detects whether a subset $J\in \J$ satisfies $J\in \sJ(A;C)$ or finds a certificate of non-relative-surjectivity, that is, a set $I\in 2^{\{1,\dots,p\}}\setminus \sJ(A;C)$ such that $I\subseteq J$.  This step requires solving the following problem
\begin{equation}\label{eq.lp.gral}
\min\{\|A\transp v + C_J\transp z\|^*: v\in A\R^n, z\in \R^J_+, \|(v,z)\|^* = 1\}.
\end{equation}
Observe that $J\in \sJ(A;C)$ if and only if the optimal value of~\eqref{eq.lp.gral} is positive. Thus, the minimization problem~\eqref{eq.lp.gral} either detects that $J\in \sJ(A;C)$ and computes 
$1/H_J(A;C)$ when its optimal value is positive, or detects that $J\not \in \sJ(A;C)$ and finds $z\in \R^J_+\setminus\{0\}$ such that $A\transp v + C_J\transp z = 0$.  In the latter case, the set $I(z):=\{i\in J: z_i > 0\}$ is a certificate of non--relative-surjectivity for $J$.

The tractability of~\eqref{eq.lp.gral} is a bit more nuanced than that of~\eqref{eq.lp} due to the presence of the unconstrained variables $v\in \R^m$.  The following easier problem allows us to determine whether the optimal value of~\eqref{eq.lp.gral}  is positive, that is, whether $J\in \sJ(A;C)$.  This is the most critical information about~\eqref{eq.lp.gral}  used in Algorithm~\ref{alg:bb.gral}
\begin{equation}\label{eq.lp.gral.easier}
\min\{\|A\transp v + C_J\transp z\|^*: v\in A\R^n, z\in \R^J_+, \|z\|^* = 1\}.
\end{equation}
Problem~\eqref{eq.lp.gral.easier} is a convex optimization problem when $\R^{p}$ is endowed with the $\ell_\infty$ norm.  
It is evident that the optimal value of~\eqref{eq.lp.gral} is zero if and only if the optimal value of~\eqref{eq.lp.gral.easier} is zero.  
Thus for the purpose of solving the main computational challenge in computing $H(A;C)$, that is, finding the collections $\F$ and $\I$, Algorithm~\ref{alg:bb.gral} can rely on the easier problem~\eqref{eq.lp.gral.easier} in place of~\eqref{eq.lp.gral}.  Nonetheless,~\eqref{eq.lp.gral} needs to be solved or estimated for the purpose of computing or estimating the value $H(A;C)$.   

\medskip

When $\R^{m+p}$ is endowed with the $\ell_1$-norm,~\eqref{eq.lp.gral} can be solved by solving $2m+|J|$ convex optimization problems.   In this case $\|(v,z)\|^* = \|(v,z)\|_\infty = \max\left(\max_{i=1,\dots,m} |v_i|,\max_{j\in J}|z_j|\right)$ and so
\begin{align*}
\min &\{\|A\transp v + C_J\transp z\|^*: v\in \R^m, z\in \R^J_+, \|(v,z)\|^* 
= 1\} \\ &
= \min\left(\begin{array}{l} 
\dmin_{i =1,\dots,m} \;  \min\{\|A\transp v + C_J\transp z\|^*: v \in A\R^n, z\in \R^J_+,  \|(v,z)\|_\infty \le 1, \, v_i = 1\}, \\
\dmin_{i =1,\dots,m} \;  \min\{\|A\transp v + C_J\transp z\|^*: v \in A\R^n, z\in \R^J_+,  \|(v,z)\|_\infty \le 1, \, v_i = -1\},\\
\dmin_{j\in J} \;  \min\{\|A\transp v + C_J\transp z\|^*: v \in A\R^n, z\in \R^J_+, \|(v,z)\|_\infty\le 1, \, z_j = 1\}
\end{array}
\right).
\end{align*}

Section~\ref{sec.euclidean} describes a more involved approach to estimate the optimal value of~\eqref{eq.lp.gral} when  $\R^n$ and $\R^m$ are endowed with the $\ell_2$ norm.

\begin{algorithm}
  \caption{
   Computation of collections of certificates $\F ,\;\I$ and constant $H(A;C)$ }
\label{alg:bb.gral}
\begin{algorithmic}[1]
\State {\bf input} $A \in \R^{m \times n}, \, C\in \R^{p\times n}$
\State Let $\F := \emptyset, \;\I := \emptyset, \; \J:=\{\{1,\dots,p\}\}, H(A;C):= 0$
\While {$\mathcal J \ne \emptyset$}
	\State Pick $J \in \J$ 
	 and let $(v,z)$ solve~\eqref{eq.lp.gral} to detect whether $J\in \sJ(A;C)$
		\If {$\|A\transp v + C_J\transp z\|^* > 0$}
			\State 
			$\F := \F \cup \{J\},\hat \J := \{\hat J \in \J: \hat J \subseteq J\},H(A;C) := \max\left\{H(A;C),\frac{1}{\|A\transp v+C_J\transp z\|^*}\right\}$
			\State Let $\J:=\J\setminus\hat \J$ 
		\Else
			\State 
			 Let $\I:= \I \cup\{I(v)\}$, $\hat \J := \left\{\hat J\in \J: I(z) \subseteq \hat J\right\}$
			 \State Let $\bar \J := \left\{\hat J\setminus \{i\}: \hat J \in \hat \J, i\in I(v), \hat J\setminus \{i\} \not \subseteq F \text{ for all } F\in \F\right\}$ 
			 \State Let $\J:= (\J \setminus \hat \J) \cup \bar \J$		\EndIf
		\EndWhile
\State \Return $\F, \, \I, \, H(A;C)$
\end{algorithmic}
\end{algorithm}

\subsection{Estimating~\eqref{eq.lp.gral} for Euclidean norms}
\label{sec.euclidean}

Throughout this subsection suppose that $\R^n$ and $\R^{m+p}$ are endowed with the $\ell_2$ norm and $J\subseteq\{1,\dots,p\}$ is fixed.  We next describe a procedure to compute lower and upper bounds on~\eqref{eq.lp.gral} within a factor 
$(4p+9)$ of each other by relying on a suitably constructed self-concordant barrier function.  
We concentrate on the case when $J\in \sJ(A;C)$ as otherwise~\eqref{eq.lp.gral.easier} can easily detect that $J \not \in \sJ(A;C)$.  By Proposition~\ref{prop.Hoffman}  the optimal value  of~\eqref{eq.lp.gral} equals
\begin{equation}\label{eq.HJ.dist}
\frac{1}{H_J(A;C)} = 
\max\{r: (y,w) \in (A\R^n)\times \R^J, \|(y,w)\|_2 \le r \Rightarrow (y,w) \in 
\D \}
\end{equation}
where $\D = \{(Ax,C_Jx + s): x\in \R^n, \, s\in \R^J_+, \|x\|_2 \le 1\}$. Equation~\eqref{eq.HJ.dist} has the following geometric interpretation: $1/H_J(A;C)$ is the distance from the origin to the relative boundary of $\D$.

Let $f(x,s):= -\log(1-\|x\|_2^2) - \dsum_{j=1}^p \log(s_j)$ and define 
$F: \ri(\D) \rightarrow \R$ as follows
\begin{equation}\label{eq.implicit}
\begin{array}{rl}F(y,w):= \dmin_{x,s} & f(x,s) \\
& Ax = y \\
& C_Jx + s = w. 
\end{array}
\end{equation}
From~\cite[Proposition 5.1.5]{NestN94} it follows that the function $F$ constructed in~\eqref{eq.implicit} is a $(p+2)$-self-concordant barrier function for $\D$.  A straightforward calculus argument shows that \begin{equation}\label{eq.dykin}
\mathcal E := \{d \in(A\R^n) \times \R^J: \ip{\nabla^2F(0,0) d}{d} \le 1 \} 
=  \{M^{1/2}d: d \in (A\R^n) \times \R^J, \|d\|\le 1\}, 
\end{equation}
where\begin{equation}
\label{eq.hessian}
M:= \matr{A & 0 \\ C_J & I} \nabla^2 f(\bar x, \bar s)^{-1} \matr{A & 0 \\ C_J & I}\transp
\end{equation}
and $(\bar x, \bar s)$ is the solution to~\eqref{eq.implicit} for $(y,w):= (0,0) \in \ri(\D)$.

The ellipsoid $\mathcal E$ in~\eqref{eq.dykin} is the Dikin ellipsoid in $(A\R^n)\times \R^J$ associated to $F$ and centered at $(0,0)$.  Therefore from the properties of self-concordant barriers~\cite{NestN94,Rene01} it follows that $\mathcal E \subseteq \D$ and 
$\{d\in \D: \ip{\nabla F(0,0)}{d} \ge 0 \} \subseteq (4p+9) \cdot \mathcal E $.  These two properties and~\eqref{eq.HJ.dist} imply that
\[
\sigma_{\min}(M^{1/2}) \le \frac{1}{H_J(A;C)} \le (4p+9) \cdot \sigma_{\min}(M^{1/2})
\]
where $\sigma_{\min}(M^{1/2})$ denotes the smallest positive singular value of $M^{1/2}$.

\medskip

We thus have the following procedure to estimate~\eqref{eq.lp.gral}: First, solve \eqref{eq.lp.gral.easier}.  If this optimal value is zero then the optimal value of~\eqref{eq.lp.gral} is zero as well.  Otherwise, let $(\bar x, \bar s)$ solve~\eqref{eq.implicit} for $(y,w) := (0,0)$ and let $M$ be as in~\eqref{eq.hessian}.  The values $\sigma_{\min}(M^{1/2})$ and $(4p+9)\cdot \sigma_{\min}(M^{1/2})$ are respectively a lower bound and an upper bound on the optimal value $1/H_J(A;C)$ of~\eqref{eq.lp.gral}.

\section{A Hoffman constant for polyhedral sublinear mappings}
\label{sec.proof}
We next present a characterization of the Hoffman constant for polyhedral sublinear mappings when the residual is known to intersect a predefined linear subspace. To that end, we will make extensive use of the following correspondence between polyhedral sublinear mappings and polyhedral cones.

A set-valued mapping $\Phi:\R^n \rightrightarrows\R^m$ is a {\em polyhedral sublinear mapping} if \[\graph(\Phi) = \{(x,y): y\in \Phi(x)\} \subseteq \R^n \times \R^m\] is a polyhedral cone.  Conversely, if $K \subseteq \R^n \times \R^m  $ is a polyhedral convex cone then the set-valued mapping $\Phi_K: \R^n \rightrightarrows\R^m$ defined via \[y \in \Phi_K(x) \Leftrightarrow (x,y) \in K\] is a polyhedral sublinear mapping since $\graph(\Phi_K) = K$ by construction.

Let $\Phi:\R^n \rightrightarrows\R^m$ be a  polyhedral sublinear mapping.  The domain, image, and norm of $\Phi$ are defined as follows:
\begin{align*}
\dom(\Phi) &= \{x\in \R^n: (x,y) \in \graph(\Phi) \text{ for some } y\in \R^m\},\\
\Image(\Phi) &= \{y\in \R^n: (x,y) \in \graph(\Phi) \text{ for some } x\in \R^n\},\\
\|\Phi\| &= \dmax_{x\in \dom(\Phi) \atop \|x\|\le 1} \min_{y\in \Phi(x)} \|y\|.
\end{align*}
In particular, the norm of the inverse mapping $\Phi^{-1}:\R^m\rightrightarrows \R^n$ is
\[
\|\Phi^{-1}\| = \dmax_{y\in \dom(\Phi^{-1})\atop \|y\|\le 1} \min_{x\in \Phi^{-1}(y)} \|x\|= \dmax_{y\in \Image(\Phi)\atop \|y\|\le 1} \min_{x\in \Phi^{-1}(y)} \|x\|.
\]
We will rely on the following more general concept of norm.  Let $\Phi:\R^n \rightrightarrows\R^m$ be a polyhedral sublinear mapping and $\L\subseteq\R^m$ be a linear subspace.  Let
\[
\| \Phi^{-1} \vert {\L}\| := \dmax_{y\in \Image(\Phi) \cap \L\atop \|y\|\le 1} \min_{x\in \Phi^{-1}(y)} \|x\|.
\]
It is easy to see that $\|\Phi^{-1}\vert {\L}\|$ is finite if  $\Phi:\R^n \rightrightarrows\R^m$ is a polyhedral sublinear mapping and $\L\subseteq \R^m$ is a linear subspace.

For $b\in \R^m$ and $S\subseteq \R^m$ define
\[
\dist_{\L}(b,S) = \inf\{\|b-y\|: y\in S, b-y\in \L\}.
\]
Observe that $\dist_{\L}(b,S) < \infty$ if and only if $(S-b)\cap \L \ne \emptyset$.  Furthermore, observe that $ \| \Phi^{-1} \vert {\L}\| =
\| \Phi^{-1} \|$ and  $\dist_{\L}(b,S)=\dist(b,S)$ when $\L=\R^m$.

Let $K\subseteq \R^n \times \R^m  $ be a polyhedral convex cone.  Let $\mathcal T(K):=\{T_K(u,v): (u,v)\in K\}$ where $T_K(u,v)$  denotes the {\em tangent} cone to $K$ at the point $(u,v)\in K$, that is,
\[
T_K(u,v) = \{(x,y) \in \R^n\times \R^m: (u,v) + t(x,y) \in K \; \text{ for some } t > 0\}.
\]
Observe that since $K$ is polyhedral the collection of tangent cones $\mathcal T(K)$ is finite. 

Recall that a polyhedral sublinear mapping $\Phi:\R^n \rightrightarrows\R^m$  is {\em relatively surjective} if $\Image(\Phi) = \Phi(\R^n)\subseteq\R^m$ is a linear subspace.
Given a polyhedral sublinear mapping $\Phi:\R^n \rightrightarrows\R^m$  let  
\[
\S(\Phi):=\{T\in \T(\graph(\Phi)): \Phi_T \; \text{ is relatively surjective}\}
\]
and
\[
\H(\Phi \vert \L):=\max_{T \in \S(\Phi)} \|\Phi_T^{-1}\vert {\L}\|.
\]

\begin{theorem}~\label{thm.main}  Let $\Phi:\R^n \rightrightarrows\R^m$ be a polyhedral sublinear mapping and $\L\subseteq\R^m$ be a linear subspace.  Then for all $b \in \Image(\Phi)$ and $u\in\dom(\Phi)$
\begin{equation}\label{eq.Hoffman.bound.symm}
\dist(u,\Phi^{-1}(b))\le \H(\Phi \vert \L)\cdot\dist_{\L}(b,\Phi(u)).
\end{equation}
Furthermore, the bound~\eqref{eq.Hoffman.bound.symm} is tight: If $\H(\Phi \vert \L) > 0$ then there exist $b \in \Image(\Phi)$ and  $u\in\dom(\Phi)$ such that $0 < \dist_{\L}(b,\Phi(u)) < \infty$ and
\[
\dist(u,\Phi^{-1}(b)) = \H(\Phi \vert \L)\cdot\dist_{\L}(b,\Phi(u)).
\]
\end{theorem}

The following lemma is the main technical component in the proof of Theorem~\ref{thm.main}. We defer its proof to the end of this section.

\begin{lemma}\label{lemma.rel.surj}
Let $\Phi:\R^n \rightrightarrows\R^m$ be a polyhedral sublinear mapping and $\L \subseteq \R^m$ be a linear subspace.  Then
\[
\dmax_{T\in \T(\graph(\Phi))}\|\Phi_T^{-1}\vert {\L}\| =  \max_{T \in \S(\Phi)} \|\Phi_T^{-1}\vert {\L}\|.
\]
\end{lemma}

\medskip

\begin{proof}[Proof of Theorem~\ref{thm.main}]
Assume  that $b-v\in \L$ for some $v\in \Phi(u)$ as otherwise the right-hand-side in
\eqref{eq.Hoffman.bound.symm} is $+\infty$ and \eqref{eq.Hoffman.bound.symm} trivially holds.  We will prove the following equivalent statement to~\eqref{eq.Hoffman.bound.symm}:  For all $b \in \Image(\Phi)$ and $(u,v)\in\graph(\Phi)$ with $b-v \in \L$
\begin{equation}\label{eq.Hoffman.bound}
\dist(u,\Phi^{-1}(b))\le \H(\Phi \vert \L)\cdot\|b-v\|.
\end{equation}
To ease notation, let $K:=\graph(\Phi)$ so in particular $\Phi = \Phi_K$.  We will use the following consequence of Lemma~\ref{lemma.rel.surj}: $\|\Phi_T^{-1}\vert {\L}\| \le \H(\Phi \vert \L)$ for all $T\in \T(K)$.

Assume $b-v\ne 0$ as otherwise there is nothing to show.  
We proceed by contradiction.   Suppose $b \in \Image(\Phi)$ and $(u,v) \in K$ are such that $b-v\in \L$ and
\begin{equation}\label{eq.contra}
\vertiii{x-u} > \H(\Phi \vert \L) \cdot \|b-v\|
\end{equation}
for all $x$ such that $(x,b)\in K$.
Let
$
d:= \frac{b-v}{\|b-v\|} \in \L
$
 and consider the optimization problem
\begin{equation}\label{eq.opt.prob}
\begin{array}{rl}
\displaystyle\max_{w,t} & t \\
& (u+w,v+td) \in K, \\
& \|w\| \le \H(\Phi \vert \L) \cdot t.
\end{array}
\end{equation}
Since $b\in \Image(\Phi)=\Image(\Phi_K)$ it follows that  
$d = (b-v)/\|b-v\| \in \Image(\Phi_{T_K(u,v)})\cap \L$.
Hence there exists $(z,d) \in T_K(u,v)$ with $\|z\|\le \|\Phi_{T_K(u,v)}^{-1}\vert {\L} \| \le \H(\Phi \vert \L)$.  Since $K$ is polyhedral, for $t > 0$ sufficiently small $(u+tz,v+td) \in K$ and so $(w,t) := (tz,t)$ is feasible for problem~\eqref{eq.opt.prob}.  Let
$$C:=\{(w,t) \in \R^n \times \R_+: (w,t) \text{ is feasible for }~\eqref{eq.opt.prob} \}.$$
Assumption~\eqref{eq.contra} implies that $t < \|b-v\|$ for all $(w,t)\in C$.  In addition, since $K$ is polyhedral, it follows that $C$ is compact. Therefore~\eqref{eq.opt.prob} has an optimal solution $(\bar w,\bar t)$ with $0<\bar t < \|b-v\|.$

Let $(u',v'):= (u + \bar w,v+\bar t d) \in K$.
Consider the modification of~\eqref{eq.opt.prob} obtained by replacing $(u,v)$ with $(u',v')$, namely
\begin{equation}\label{eq.opt.prob.mod}
\begin{array}{rl}
\displaystyle\max_{w' ,t'} & t' \\
& (u'+w',v'+t'd) \in K, \\
& \vertiii{w'} \le \H(\Phi \vert \L)\cdot t'.
\end{array}
\end{equation}
Observe that
$
b - v' = b-v-\bar t d = (\|b-v\| - \bar t)d \ne 0.
$
Again since $b\in \Image(\Phi)$ it follows that $d= \frac{b-v'}{\|b-v'\|}  \in\Image(\Phi_{T_K(u',v')})\cap \L$.  Hence there exists $(z',d)\in T_K(u',v')$ such that  $\|z'\|\le \|\Phi_{T_K(u',v')}^{-1}\vert {\L}\| \le \H(\Phi \vert \L)$.  Therefore,~\eqref{eq.opt.prob.mod} has a feasible point $(w',t') = (t'z',t')$ with $t' > 0$.  In particular $(u'+w',v'+t'd) = (u+\bar w + w', v + (\bar t + t')d) \in K$ with $\|\bar w + w'\| \le \|\bar w\| + \|w'\| \le \H(\Phi \vert \L) \cdot(\bar t +t')$ and $\bar t + t' >\bar t$.  This contradicts the optimality of $(\bar w,\bar t)$ for~\eqref{eq.opt.prob}.

To show that the bound is tight, suppose $\H(\Phi \vert \L) = \|\Phi_T^{-1}\vert {\L}\| > 0$ for some $T\in \S(\Phi) \subseteq \T(K)$. The construction of $\vertiii{\Phi_T^{-1}\vert {\L}}$ implies that there exists $d \in \L$ with  $\|d\|=1$ such that the problem
\begin{equation}\label{eq.tight}
\begin{array}{rl}
\displaystyle\min_{z} & \vertiii{z} \\
& (z,d) \in T
\end{array}
\end{equation}
is feasible and has an optimal solution $\bar z$ with $\|\bar z\| = \|\Phi_T^{-1}\vert {\L} \| = \H(\Phi \vert \L)>0$.  Let $(u,v)\in K$ be such that $T = T_K(u,v)$.   Let $b:=v+td$ where $t > 0$ is small enough so that $(u,v) + t(\bar z,d)\in K$.  Observe that
$b \in \Image(\Phi)$ and $b - v = t d \ne 0$.
To finish, notice that if $x\in \Phi^{-1}(b)$ then $(x-u,b-v) = (x-u,td) \in T_K(u,v) = T$.  The optimality of $\bar z$ then implies that
\[
\|x-u\| \ge \H(\Phi \vert \L) \cdot t = \H(\Phi \vert \L)\cdot \|b-v\|.
\]
Since this holds for all $x\in \Phi^{-1}(b)$ and $b-v\in \L\setminus \{0\}$, it follows that 
$\dist(u,\Phi^{-1}(b)) \ge \H(\Phi \vert \L)\cdot \|b-v\| \ge \H(\Phi \vert \L)\cdot\dist_{\L}(b,\Phi(u))>0.$
\end{proof}

The proof of Lemma~\ref{lemma.rel.surj} relies on a  convex duality construction.  
In each of $\R^n$ and $\R^m$ let~$\|\cdot\|^*$ denote the dual norm of $\|\cdot\|$, that is, for $u\in \R^n$ and $v\in\R^m$
\[
\|u\|^*:=\dmax_{x\in \R^n \atop \|x\|\le 1} \ip{u}{x}\, \text{ and } \, \|v\|^*:=\dmax_{y\in \R^m \atop \|y\|\le 1} \ip{v}{y}.
\]
Given a cone $K\subseteq\R^n \times \R^n$, let $K^* \subseteq\R^n \times \R^m$ denote its dual cone, that is,
\[
K^*:= \{(u,v) \in \R^n\times \R^m: \ip{u}{x}+\ip{v}{y} \ge 0 \text{ for all } (x,y) \in  K\}.
\]
Given a sublinear mapping $\Phi: \R^n\rightrightarrows\R^m$, let $\Phi^*: \R^m\rightrightarrows \R^n$ denote its {\em upper adjoint,} that is
\[
u\in \Phi^*(v) \Leftrightarrow \ip{u}{x}\le \ip{v}{y} \text{ for all } (x,y) \in \graph(\Phi).
\]
Equivalently, $u \in \Phi^*(v) \Leftrightarrow (-u,v) \in \graph(\Phi)^*$.

Observe that  for a polyhedral convex cone $T\subseteq\R^n \times \R^m$ and a linear subspace $\L\subseteq \R^m$
\[
\begin{array}{rl}
\|\Phi_T^{-1}\vert {\L}\| = \dmax_{y} & \vertiii{\Phi_T^{-1}(y)}\\
&  y \in \Image(\Phi_T) \cap \L,\\
& \|y\|\le 1,
\end{array}
\]
where
\begin{equation}\label{primal.Hoffman}
\begin{array}{rl}
\vertiii{\Phi_T^{-1}(y)} := \dmin_{x} & \|x\| \\
& (x,y) \in T.
\end{array}
\end{equation}
By convex duality it follows that
\begin{equation}\label{dual.Hoffman}
\begin{array}{rl}
\vertiii{\Phi_T^{-1}(y)} =  \dmax_{u,v} & -\ip{v}{y},\\
& \|u\|^* \le 1, \\
& (u,v)\in T^*.
\end{array}
\end{equation}
Therefore when $T$ is a polyhedral cone
\begin{equation}\label{eq.norm.Hoffman}
\begin{array}{rl}
\vertiii{\Phi_T^{-1}\vert {\L}} = \dmax_{u,v,y} &  -\ip{v}{y} \\
&  y \in \Image(\Phi_T) \cap \L,\\
& \|y\|\le 1, \\
& \|u\|^* \le 1, \\
& (u,v)\in T^*.
\end{array}
\end{equation}

For a linear subspace $\L\subseteq \R^m$ let $\Pi_{\L}: \R^m \rightarrow \L$ denote the orthogonal projection onto $\L$.  The following proposition is in the same spirit as  Borwein's norm-duality Theorem~\cite{Borw83}.

\begin{proposition}\label{prop.Hoffman.surj}
Let $\Phi:\R^n\rightrightarrows\R^m$ be a  polyhedral sublinear mapping
and $\L\subseteq \R^m$ be a linear subspace.  If $\Phi$ is relatively surjective then
\[
\H(\Phi \vert \L) = \|\Phi^{-1}\vert {\L}\| = \dmax_{u\in \Phi^*(v)\atop \|u\|^*\le 1} \|\Pi_{\Image(\Phi) \cap \L}(v)\|^* =  
\frac{1}
{\dmin_{u\in \Phi^*(v)\atop\|\Pi_{\Image(\Phi) \cap \L}(v)\|^*=1}\|u\|^*}.
\]
\end{proposition}
\begin{proof}
Since $\graph(\Phi) \subseteq T$ for all $T\in \mathcal T(\graph(\Phi))$ and $\Phi$ is relatively surjective, it  follows that $\vertiii{\Phi_T^{-1}\vert {\L}} \le \vertiii{\Phi^{-1}\vert {\L}}$ for all $T\in \mathcal T(\graph(\Phi))$.
 Consequently
$ \H(\Phi \vert \L) = \|\Phi^{-1}\vert {\L}\|.$
Furthermore, since $\Phi$ is relatively surjective, from~\eqref{eq.norm.Hoffman} it follows that
\[
\begin{array}{rl}
\vertiii{\Phi^{-1}\vert {\L}} = \dmax_{u,v} &  \|\Pi_{\Image(\Phi) \cap \L}(v)\|^* \\
& \|u\|^* \le 1, \\
& u\in \Phi^*(v).
\end{array}
\]
The latter quantity is evidently the same as 
$\dfrac{1}
{\dmin_{u\in \Phi^*(v)\atop\|\Pi_{\Image(\Phi) \cap \L}(v)\|^*=1}\|u\|^*}.$
\end{proof}

We will rely on the following equivalence between {\em surjectivity} and {\em non-singularity} of sublinear mappings.   A standard convex separation argument shows that  a closed sublinear mapping $\Phi:\R^n \rightrightarrows \R^m$ is  surjective if and only if
\begin{equation}\label{eq.non.sing}
(0,v) \in \graph(\Phi)^* \Rightarrow v=0.
\end{equation}
Condition~\eqref{eq.non.sing} is a kind of {\em non-singularity} of  $\Phi^*$ as it can be rephrased as $0 \in \Phi^*(v) \Rightarrow v=0.$

\begin{proof}[Proof of Lemma~\ref{lemma.rel.surj}] Without loss of generality assume $\linspan(\Image(\Phi)) = \R^m$ as otherwise we can work with the restriction of $\Phi$ as a mapping from $\R^n$ to $\linspan(\Image(\Phi))$.  
To ease notation let $K:=\graph(\Phi)$.  We need to show that
\[
\dmax_{T\in \T(K)}  \|\Phi_T^{-1}\vert {\L}\| =
\dmax_{T\in \S(\Phi)}  \|\Phi_T^{-1}\vert {\L}\|.
\]
By construction, it is immediate that
\[
\dmax_{T\in \mathcal T(K)}  \|\Phi_T^{-1}\vert {\L}\| \ge
\dmax_{T\in \S(\Phi)}  \|\Phi_T^{-1}\vert {\L}\|.
\]
To prove the reverse inequality let $T \in \mathcal T(K)$ be fixed and let $(\bar u,\bar v, \bar y)$ attain the optimal value $\|\Phi_T^{-1}\vert {\L}\|$ in~\eqref{eq.norm.Hoffman}.
Let $\bar F$ be the minimal face of $K^*$ containing $(\bar u, \bar v)$ and $\bar T := \bar F^*  \in \mathcal T(K)$.   As we detail below, $(\bar u,\bar v, \bar y)$ can be chosen so that $\Phi_{\bar T}$ is surjective.
If $\|\Phi_T^{-1}\vert {\L}\| = 0$ then it trivially follows that $\|\Phi_T^{-1}\vert {\L}\| \le \|\Phi_{\bar T}^{-1}\vert {\L}\|.$   Otherwise, 
since $\|\bar y\|\le 1$ and $\bar y \in \L$ we have
\[
\|\Phi_T^{-1}\vert {\L}\| = -\ip{\bar v}{ \bar y} \le \|\Pi_{\L}(\bar v)\|^*.
\]
Since $(\bar u,\bar v) \in T^* = \graph(\Phi_{\bar T})^*$ and $\|\bar u\|^* \le 1$,
 Proposition~\ref{prop.Hoffman.surj} yields
\[
\|\Phi_T^{-1}\vert {\L}\| \le \|\Pi_{\L}(\bar v)\|^* \le \|\Phi_{\bar T}^{-1}\vert {\L}\|.
\]
In either case $\|\Phi_T^{-1}\vert {\L}\| \le \|\Phi_{\bar T}^{-1}\vert {\L}\|$ where $\bar T \in \S(\Phi)$.  Since this holds for any fixed $T \in\mathcal T(K)$, it follows that
\[
\dmax_{T\in \mathcal T(K)} \vertiii{\Phi_T^{-1}\vert {\L}} \le  \dmax_{\bar T\in \S(\Phi)}  \|\Phi_{\bar T}^{-1}\vert {\L}\|. 
\]

It remains to show that $(\bar u,\bar v, \bar y)$ can be chosen so that $\Phi_{\bar T}$ is surjective, where $\bar T = \bar F^*$ and $\bar F$ is the minimal face of $K^*$ containing $(\bar u,\bar v)$.  To that end, pick a 
solution $(\bar u,\bar v,\bar y)$ to~\eqref{eq.norm.Hoffman} and consider the set
$$
V:=\{v \in \R^m: \ip{v}{\bar y} = \ip{\bar v}{\bar y}, \, (\bar u,v)\in T^*\}.$$
In other words, $V$ is the projection of the set of optimal solutions  to~\eqref{eq.norm.Hoffman} of the form $(\bar u, v, \bar y)$.  Since $T$ is polyhedral, so is $T^*$ and thus $V$ is a polyhedron.
Furthermore, $V$ must have at least one extreme point.  Otherwise there exist $\hat v\in V$ and a nonzero $\tilde v\in \R^m$ such that $\hat v + t\tilde v \in V$ for all $t \in \R$.   In particular, $(\bar u, \hat v + t\tilde v) \in T^*$ for all $t \in \R$ and thus both  $(0,\tilde v) \in T^*\subseteq K^*$ and $-(0,\tilde v)\in T^*\subseteq K^*$.
The latter in turn implies $\Image(\Phi) \subseteq \{y\in \R^m: \ip{\tilde v}{y} =0\}$ contradicting the assumption $\linspan(\Image(\Phi)) =\R^m$.  By replacing $\bar v$ if necessary, we can assume that
$\bar v$ is an extreme point of $V$. We claim that the minimal face $\bar F$ of $K^*$ containing $(\bar u,\bar v)$ satisfies
\[
(0,v') \in \bar F = \bar T^* \Rightarrow v' = 0
\]
thereby establishing the surjectivity of $\Phi_{\bar T}$ (cf., \eqref{eq.non.sing}).
To prove this claim, proceed by contradiction.  Assume $(0,v') \in \bar F$ for some nonzero $v'\in \R^m$.
The choice of $\bar F$ ensures that $(\bar u,\bar v)$ lies in the relative interior of $\bar F$ and thus for $t>0$ sufficiently small both $(\bar u,\bar v + tv')\in \bar F\subseteq T^*$ and
$(\bar u,\bar v - tv')\in \bar F\subseteq T^*$.  The optimality of $(\bar u,\bar v,\bar y)$  implies that both $\ip{\bar v+tv'}{\bar y } \ge \ip{\bar v}{ \bar y}$ and $\ip{\bar v -tv'}{\bar y} \ge \ip{\bar v}{ \bar y}$ and so
$\ip{v'}{ \bar y} = 0$.  Thus both $\bar v + tv' \in V$ and $\bar v -tv'\in V$ with $tv'\ne 0$ thereby contradicting the assumption that $\bar v$ is an extreme point of $V$.
\end{proof}

\section{Proofs of propositions in Section~\ref{sec.Hoffman}}\label{sec.proof.Hoffman}

\begin{proof}[Proof of Proposition~\ref{prop.Hoffman.gral.rest}]
Let $\Phi:\R^n \rightrightarrows \R^m$ be defined by $\Phi(x):= Ax+\R^n_+$ and $\mathcal L:=\{y\in \R^m: y_{L^c} = 0\}$. Observe that for this $\Phi$ we have
\[
\graph(\Phi) = \{(x,Ax+s) \in \R^n \times \R^m: s \ge 0\}.
\]
Hence $\T(\graph(\Phi)) = \{T_J: J \subseteq \{1,\dots,m\}\}$ where
\[
T_J = \{(x,Ax+s) \in \R^n \times \R^m: s_J \ge 0\}.
\]
Furthermore, for $T_J$ as above the mapping $\Phi_{T_J}:\R^n \rightrightarrows \R^m$ is defined by
\[
\Phi_{T_J}(x) = \{Ax + s: s_J \ge 0\}.
\]
Therefore, $\Phi_{T_J}$ is relatively surjective if and only if $J$ is $A$-surjective.  In other words, $T_J \in \S(\Phi) \Leftrightarrow J \in \sJ(A)$ and in that case \[
\|\Phi_{T_J}^{-1}\vert {\mathcal L}\|  = \dmax_{y \in \R^L \atop \|y\| = 1} \min_{x\in\R^n\atop A_J x\le y_J}\|x\| 
= H_J(A\vert L).
\]
To finish, apply Theorem~\ref{thm.main} to $\Phi$ and $\mathcal L$.
\end{proof}

\begin{proof}[Proof of Proposition~\ref{prop.Hoffman.A.surj.rest}]
Let $\mathcal L:=\{y\in \R^m: y_{L^c} = 0\}$.  If $J \in \sJ(A)$ then the polyhedral sublinear mapping $\Phi:\R^n \rightrightarrows \R^m$ defined via
\[
x\mapsto Ax + \{s\in \R^m: s_J \ge 0\}
\]
is surjective.  Thus Proposition~\ref{prop.Hoffman.surj} yields
\[
H_J(A\vert L) = \|\Phi^{-1}\vert {\mathcal L}\| = \frac{1}{\dmin_{(u,v) \in \graph(\Phi)^*\atop\|\Pi_{\mathcal L}(v)\|^*=1} \|u\|^*}. 
\]
To get~\eqref{eq.HA.J.L}, observe that $u\in \Phi^*(v)$ if and only if 
$u = A\transp v, v_J \ge 0,$ and $v_{J^c} = 0$, and when that is the case $\Pi_{\mathcal L}(v) = v_{J\cap L}$.
Finally observe that~\eqref{eq.HA.L} readily follows form~\eqref{eq.HA.J.L}.
\end{proof}

Proposition~\ref{prop.Hoffman.gral} and Proposition~\ref{prop.Hoffman.A.surj} follow as special cases of 
Proposition~\ref{prop.Hoffman.gral.rest} and Proposition~\ref{prop.Hoffman.A.surj.rest}
by taking $L = \{1,\dots,m\}$.  The proofs of the remaining propositions are similar to the proofs of Proposition~\ref{prop.Hoffman.gral.rest} and Proposition~\ref{prop.Hoffman.A.surj.rest}.  

\begin{proof}[Proof of Proposition~\ref{prop.Hoffman}]
Let $\Phi:\R^n \rightrightarrows \R^m\times \R^p$ be defined as 
\[
\Phi(x) := \matr{Ax\\Cx} +  \{0\}\times \R^p_+
\]
and $\L := \R^m \times \R^p$.  Then $\T(\graph(\Phi)) = \{T_J: J\subseteq\{1,\dots,p\}\}$ where
\[
T_J := \{(x,Ax,Cx+s)\in \R^n\times (\R^m\times \R^p): s_J \ge 0\}.
\]
Furthermore,
$T_J\in \S(\Phi) \Leftrightarrow J \in \sJ(A;C)$ and in that case
\[
\|\Phi_{T_J}^{-1}\vert {\L}\|=\dmax_{(y,w)\in A\R^m\times \R^p \atop \|(y,w)\|\le 1} \dmin_{x\in \R^n \atop Ax = y, C_Jx \le w_J} \|x\|.
\]
Observe that 
$
u \in \Phi_{T_J}^*(v,z) \Leftrightarrow (-u,v,z) \in T_J^*  \Leftrightarrow u = A\transp v + C_J\transp z_J, \;  z_J \ge 0,$  and $z_{J^c} = 0.$
Thus for $J\in\sJ( A;C)$
Proposition~\ref{prop.Hoffman.surj} yields
\begin{align*}
\|\Phi_{T_J}^{-1}\vert {\L}\| &= \dmax_{(v,z)\in \R^m\times\R^p_+ \atop z_{J^c} = 0, \|A\transp v +C\transp z\|^* \le 1} \|\Pi_{\Image(\Phi_{T_J})\cap\L}(v,z)\|^*  \\ &= \dmax_{(v,z)\in \R^m\times\R^p_+ \atop z_{J^c} = 0, \|A\transp v +C\transp z\|^* \le 1} \|\Pi_{A(\R^n)\times \R^p}(v,z)\|^* \\&= 
\dmax_{(v,z)\in (A\R^n)\times\R^p_+ \atop z_{J^c} = 0, \|A\transp v +C\transp z\|^* \le 1} \|(v,z)\|^*
\end{align*}
To finish, apply Theorem~\ref{thm.main}.
\end{proof}

\begin{proof}[Proof of Proposition~\ref{prop.Hoffman.std}]
The proof is identical to the proof of Proposition~\ref{prop.Hoffman} if we take $\L = \R^m \times \{0\}$ instead.
\end{proof}

\begin{proof}[Proof of Proposition~\ref{prop.equal.easy}]
The proof is identical to the proof of Proposition~\ref{prop.Hoffman} if we take $\L = \{0\} \times \R^p$ instead.
\end{proof}

\begin{proof}[Proof of Proposition~\ref{prop.facial.dist}]
The proof is similar to the proof of Proposition~\ref{prop.Hoffman}. 
Let $\Phi:\R^n \rightrightarrows \R^m \times \R \times \R^n$ be defined as
\[
\Phi(x) = \matr{\tilde A x\\ Cx} + \{0\} \times \R^n_+
\]
and $\L = \R^m\times \{0\} \times\{0\} \subseteq \R^m \times \R \times \R^n$.
Then $\T(\graph(\Phi)) = \{T_J: J\subseteq\{1,\dots,p\}\}$ where
\[
T_J := \{(x,\tilde Ax,-x+s)\in \R^n\times (\R^m\times \R\times\R^n): s_J \ge 0\}.
\]
Furthermore,
$T_J\in \S(\Phi) \Leftrightarrow J \in \sJ(\tilde A;C)$ and in that case 
\[
\Image(\Phi_{T_J})\cap\L = \left\{(y,0,0) : y = Ax, 0 = \1\transp x, \; 0=-x+s \; \text{for} \; s\in \R^n \; \text{with}\; s_J \ge 0\right\}.
\]
Thus
\[
\|\Phi_{T_J}^{-1}\vert {\L}\|=\dmax_{y\in AK_J \atop \|y\|\le 1} \dmin_{x\in K_J \atop Ax = y} \|x\|.
\]
Observe that 
$
u \in \Phi_{T_J}^*(v,t,z) \Leftrightarrow  u =  A\transp v + t\1 -z, \; z_J \ge 0,$ and $z_{J^c} = 0.$
Thus for $J\in\sJ(\tilde A;C)$  Proposition~\ref{prop.Hoffman.surj} yields
\begin{align*}
\|\Phi_{T_J}^{-1}\vert {\L}\| &= \dmax_{(v,t,z)\in \R^m\times \R\times \R^n_+ \atop z_{J^c} = 0,\|A\transp v +t\1 -z\|^* \le 1} \|\Pi_{\Image(\Phi_{T_J})\cap\L}(v,t,z)\|^*  \\ &= \dmax_{(v,t,z)\in \R^m\times \R\times \R^n_+ \atop z_{J^c} = 0, \|A\transp v +t\1 -z\|^* \le 1} \|\Pi_{(\tilde A(\R^n)\times \R^n)\cap (\R^m\times \{0\}\times \{0\})}(v,t,z)\|^*  \\
&= \dmax_{(v,t)\in\tilde A\R^n, z\in \R^n_+ \atop z_{J^c} = 0, \|A\transp v +t\1 -z\|^* \le 1} \|\Pi_{\tilde A(\R^n)\cap (\R^m\times \{0\})}(v,t)\|^*  \\
&= 
\dmax_{(v,t)\in\tilde A\R^n, z\in \R^n_+ \atop z_{J^c} = 0, \|A\transp v + t\1 -z\|^* \le 1} \|\Pi_{L_A}(v)\|^*.
\end{align*}
To finish, apply Theorem~\ref{thm.main}.

\end{proof}

\section*{Acknowledgements}

Javier Pe\~na's research has been funded by NSF grant CMMI-1534850.


\end{document}